\documentclass[11pt,a4paper]{article}

\usepackage[math]{cellspace}
\usepackage[normalem]{ulem}
\usepackage{empheq}
\usepackage{psfrag}
\usepackage{amsmath,amssymb,cite}
\usepackage{latexsym}
\usepackage{graphicx}
\usepackage{lscape}
\usepackage{afterpage}
\usepackage[color=green!40]{todonotes}
\usepackage{caption}
\usepackage{subcaption}

\usepackage{hyperref}
\hypersetup{
	colorlinks=true,
	linkcolor=blue,
	filecolor=magenta,      
	urlcolor=cyan,
	citecolor=red,
}
\DeclareMathOperator*{\argmin}{argmin}
\DeclareMathOperator*{\rank}{rank}

\DeclareMathOperator*{\sgn}{sgn}

\newcommand{\ds}{\displaystyle}
\newcommand{\nexto}{\kern -0.54em}
\newcommand{\dR}{{\rm {I\ \nexto R}}}

\newcommand{\dZ}{{\cal Z \kern -0.7em Z}}
\newcommand{\dC}{{\rm\hbox{C \kern-0.8em\raise0.2ex\hbox{\vrule
				height5.4pt width0.7pt}}}}
\newcommand{\dQ}{{\rm\hbox{Q \kern-0.85em\raise0.25ex\hbox{\vrule
				height5.4pt width0.7pt}}}}
\newcommand{\proofbox}{\hspace{\fill}{$\Box$}}
\newtheorem{lemma}{Lemma}
\newtheorem{theorem}{Theorem}
\newtheorem{corollary}{Corollary}

\newtheorem{remark}{Remark}

\newenvironment{proof}{Proof.}{\proofbox}

\definecolor{myblue}{rgb}{0.9,0.9,0.98}

\begin{document}

\author{Authors}

\author{
Regina S. Burachik\thanks{Mathematics, UniSA STEM, University of South Australia, Mawson Lakes, S.A. 5095, Australia. E-mail: regina.burachik@unisa.edu.au and yalcin.kaya@unisa.edu.au\,.}
\and
C. Yal{\c c}{\i}n Kaya\footnotemark[1]
\and
Walaa M. Moursi\thanks{Department of Combinatorics and Optimization, University of Waterloo, Waterloo, Ontario N2L~3G1, Canada. E-mail: walaa.moursi@uwaterloo.ca\,.}
}

\title{\bf Infeasible and Critically Feasible Optimal Control}

\maketitle

\vspace*{-7mm}

\begin{abstract} 
{\noindent\sf
We consider optimal control problems involving two constraint sets: one comprised of linear ordinary differential equations with the initial and terminal states specified and the other defined by the control variables constrained by simple bounds.  When the intersection of these two sets is empty, typically because the bounds on the control variables are too tight, the problem becomes infeasible.  In this paper, we prove that, under a controllability assumption, the ``best approximation'' optimal control minimizing the distance (and thus finding the ``gap'') between the two sets is of bang--bang type, with the ``gap function'' playing the role of a switching function.  The critically feasible control solution (the case when one has the smallest control bound for which the problem is feasible) is also shown to be of bang--bang type.  We present the full analytical solution for the critically feasible problem involving the (simple but rich enough) double integrator.  We illustrate the overall results numerically on various challenging example problems.
}
\end{abstract}
\begin{verse}
\end{verse}

\pagestyle{myheadings}
\markboth{}{\sf\scriptsize Infeasible and Critically Feasible Optimal Control\ \ by R. S. Burachik, C. Y. Kaya \& W. M. Moursi}

\section{Introduction}

Optimal control problems are infinite-dimensional optimization problems, involving processes evolving with time.  Infeasibility in optimal control arises in many situations, most typically when resources for a process are overly limited: for example, insufficient amount of insecticides for a dengue epidemic~\cite{SepVasSvi2020} or a highly restricted driving motor capacity of a vehicle~\cite{XiaCasBel2021}.  Infeasibility can also arise when one aims to achieve initial or terminal states which are not realistic or when there are state constraints which are too restrictive~\cite{ZauGatMue2023}.
In this paper, we consider infeasible and critically feasible optimal control problems, where the dynamics are governed by linear ordinary differential (state) equations with initial and end states specified and the control variables constrained by simple bounds.  

Infeasibility is also widely encountered in finite-dimensional optimization problems: Error in measurements, for example the noise in images taken during computer tomography, may give rise to an inconsistent set of equations in a pertaining optimization model, making the problem infeasible~\cite{SidJorPan2013}.  In~\cite{ByrCurNoc2010}, and in its extension~\cite{BurCurWan2014}, algorithms, which incorporate sequential quadratic programming methods, are proposed for infeasible finite-dimensional nonconvex optimization problems.  Under a set of conditions, these algorithms are shown to be convergent to an infeasible stationary point, minimizing a measure of infeasibility~\cite{ByrCurNoc2010, BurCurWan2014}.

In their  paper~\cite{BauMou2017}, Bauschke and Moursi study the Douglas--Rachford (DR) algorithm for finding 
a point in the intersection of two
nonempty closed and convex sets in (possibly 
infinite-dimensional) Hilbert spaces.  They show that, for the case when the intersection of the two sets is empty, i.e., when the problem is infeasible, the DR algorithm finds a pair of points in the respective  sets which minimize (as a measure of feasibility) the distance between the two sets; in other words, the DR algorithm finds the ``gap'' between the two constraint sets assuming that the gap is attained.  
The work in \cite{BauMou2017} has been further generalized 
in \cite{BauMou2023}. Indeed, the authors in \cite{BauMou2023}
proved that under mild assumptions (see also \cite{Moursi2022}) the DR algorithm can find a 
\emph{generalized solution} (this is also known as \emph{normal solution})  
(see \cite[Definition~3.7]{BauHarMou2014}) for inconsistent convex optimization problems, i.e., when the solution set is empty.

The results in~\cite{BauMou2023} are not only applicable to infinite-dimensional problems (e.g. optimal control problems), but also if one of the constraints is {\em hard}, that is a particular constraint must be satisfied, then a solution satisfying the hard constraint, that is an {\em implementable solution}, can be returned.  In the present paper, we use the idea of the minimization of the distance between the two sets, namely finding the ``gap'' between the two sets (assuming it is attained), as our motivation in finding a best approximation solution to infeasible optimal control problems so that the optimal control we find is also implementable.

The optimal control problems we consider have two constraint sets: one involves the ODE with specified initial and end states (this set is an affine subspace, which is closed and convex) and the other involves the box constraint on the control (this set is a box, which is also closed and convex).  We pose the problem of finding a best approximation pair
to the infeasible problem as one of minimizing the distance between these two constraint sets and finding the ``gap''.  In practical optimal control problems, the control variable is expected to satisfy the simple bounds imposed on it; therefore, we regard the box as a {\em hard} constraint set.

A best approximation pair of Problem~(Pf) 
below (see equation \eqref{prob:min_dist} below)  can be expressed as a solution of a minimization problem which is strongly convex w.r.t. one the variables (see  Lemmas~\ref{lem:D1}--\ref{lem:D2}).

We prove that, under a controllability assumption, the control variable that belongs to the box and solves the best approximation problem is of {\em bang--bang type}, i.e., the value of the control variable switches between its lower and upper bounds.  Interestingly, the sign of a gap function component determines which bound value the corresponding control variable component in the box must take; in other words, a gap function component plays the role of a {\em switching function}.  We also formulate the problem of finding a critically feasible solution, i.e., a solution for the least bound on the control resulting in a nonempty intersection of the two constraint sets.  We prove that the critically feasible optimal control is also of bang--bang type.  For the case of a double integrator problem, which is often employed as part of case studies for optimal control, we derive the full analytical solution for the critically feasible optimal control problem.  

For a numerical illustration of the results, both for the critically feasible and infeasible cases, we study example problems involving (i) a double integrator, (ii) a damped oscillator and (iii) a machine tool manipulator, in the order of increasing numerical difficulty.

The paper is organized as follows.  In Section~\ref{sec:prem}, we introduce the optimal control problem and define the two constraint sets, namely the affine space and the box.  In Section~\ref{sec:best_approx}, we define the problem of best approximation, provide the maximum principle, discuss controllability and existing results, and derive the first main result of the paper on infeasible problems in Theorem~\ref{thm:gap&bbang}.  In Section~\ref{sec:crit_feas}, we introduce the concept of critical feasibility and provide the second main result in Theorem~\ref{theo:crit}. We also provide the full critically feasible solution for a problem involving the double integrator in Theorem~\ref{control_crit}.  In Section~\ref{sec:num_exp}, we carry out numerical experiments on various example problems to illustrate the results of the paper.  Finally in Section~\ref{sec:conc}, we provide concluding remarks and comment on future lines of research.

\section{Preliminaries}
\label{sec:prem}

We consider optimal control problems where the aim is to find a control $u$ which minimizes a general functional
\begin{equation}  \label{obj_fun}
\ds \int_{0}^{1} f_0(x(t),u(t),t)\, dt\,,
\end{equation}
subject to the differential equation constraints
\begin{equation}  \label{ODE}
	\dot{x}(t) = A(t)\,x(t) + B(t)\,u(t)\,,\ \ \mbox{for a.e.\ } t\in[0,1]\,,
	\end{equation}
with $\dot{x} := dx/dt$, and the boundary conditions
\begin{equation}  \label{BC}
	\varphi(x(0),x(1)) = 0\,.
\end{equation}
In the optimal control problem above the time horizon is set to be $[0,1]$, but without loss of generality it can be taken as any interval $[t_0,t_f]$, with $t_0$ and $t_f$ specified.  The integrand function $f_0:\dR^n\times\dR^m\times[0,1]\to\dR_+$ is continuous. We define the {\em state variable vector} $x:[0,1]\to\dR^n$ with $x(t) := (x_1(t)\,\ldots,x_n(t))\in\dR^n$ and the {\em control variable vector}\linebreak $u:[0,1]\to\dR^m$ with $u(t) := (u_1(t)\,\ldots,u_m(t))\in\dR^m$. 
The time-varying matrices\linebreak $A:[0,1]\to \dR^{n\times n}$ and $B:[0,1]\to\dR^{n\times m}$ are continuous.  The vector function $\varphi:\dR^{2n}\to\dR^r$, with $\varphi(x(0),x(1)) := (\varphi_1(x(0),x(1)),\ldots,\varphi_r(x(0),x(1)))\in\dR^r$, is affine.


It is realistic, especially in practical situations, to consider restrictions on the values $u$ is allowed to take.  In many applications, it is common practice to impose simple bounds on the components of $u(t)$; namely,
\begin{equation}  \label{bounds}
	\underline{a}_i(t) \le u_i(t) \le \overline{a}_i(t)\,,\ \ \mbox{for a.e.\ } t\in[0,1]\,,
\end{equation}
where, respectively, the lower and upper bound functions $\underline{a}_i,\overline{a}_i:[0,1]\to\dR$ are continuous and that $\underline{a}_i(t) < \overline{a}_i(t)$, for all $t\in[0,1]$, $i = 1,\ldots,m$.  We define for convenience $\underline{a} := (\underline{a}_1\,\ldots,\underline{a}_m)$ and $\overline{a} := (\overline{a}_1\,\ldots,\overline{a}_m)$, and write in concise form $\underline{a}(t) \le u(t) \le \overline{a}(t)$; in other words, we {\em formally} state
\begin{equation}  \label{set_bounds}
	u(t)\in U(t) := [\underline{a}(t),\overline{a}(t)]\subset\dR^m,\ \ \mbox{for a.e.\ } t\in[0,1]\,,
\end{equation}
as an expression alternative but equivalent to~\eqref{bounds}.
	
The objective functional in~\eqref{obj_fun} and the constraints in~\eqref{ODE}--\eqref{BC} and \eqref{bounds} can be put together to present the optimal control problem as follows.
\[
\mbox{(P) }\left\{\begin{array}{rl}
	\ds\min_{u(\cdot)} & \ \ \ds\int_0^1 f_0(x(t),u(t),t)\, dt \\[4mm] 
	\mbox{subject to} & \ \ \dot{x}(t) = A(t)\,x(t) + B(t)\,u(t)\,,\ \ \mbox{for a.e.\ } t\in[0,1]\,, \\[2mm]
	& \ \ \varphi(x(0),x(1)) = 0\,, \\[2mm]
	& \ \ \underline{a}_i(t) \le u_i(t) \le \overline{a}_i(t)\,,\ \ \mbox{for a.e.\ } t\in[0,1]\,,\ \ i = 1,\ldots,m\,.
	\end{array} \right.
\]
We split the constraints of Problem~(P) into two sets:
\begin{eqnarray} 
	&& {\cal A} := \big\{u\in L^{2}([0,1];\dR^m)\ |\ \exists x\in W^{1,2}([0,1];\dR^n)\mbox{ which solves } \nonumber \\[1mm]
	&&\hspace*{45mm} \dot{x}(t) = A(t)\,x(t) + B(t)\,u(t)\,,\ \ \mbox{for a.e.\ } t\in[0,1]\,, \mbox{ and }\nonumber \\[1mm]
	&&\hspace*{45mm} \varphi(x(0),x(1)) = 0 \big\}\,, \label{A} \\[2mm]
	&& {\cal B} := \big\{u\in L^{2}([0,1];\dR^m)\ |\ \underline{a}(t) \le u(t) \le \overline{a}(t)\,,\ \mbox{for a.e.\ } t\in[0,1]\big\}\,. \label{B}
\end{eqnarray}
We assume that the control system $\dot{x}(t) = A(t)x(t) + B(t)u(t)$ is {\em controllable}---See the precise definition in Section~\ref{sec:controllability}. Then there exists a (possibly not unique) $u(\cdot)$ such that, when this $u(\cdot)$ is substituted, the boundary-value problem given in ${\cal A}$ has a solution $x(\cdot)$.  In other words, ${\cal A} \neq \emptyset$.  Also, clearly, ${\cal B} \neq \emptyset$.  Recall that $\varphi$ is affine, so the constraint set~${\cal A}$ is an {\em affine subspace}. We note that by \cite[Corollary 1]{BurCalKay2024},
\begin{equation}  \label{eq:A closed}
{\cal A}\ \ \mbox{is closed}.  
\end{equation}
Given that ${\cal B}$ is a {\em box}, the constraints turn out to be two convex sets in Hilbert space.  In particular, we note that ${\cal B}$ is closed in $L^2(0,1;\dR^m)$.  
It will be convenient to use the expression
\[
B(t)u(t) = \sum_{i=1}^m b_i(t)\,u_i(t)\,,
\]
where $b_i(t)$ is the $i$th column of the matrix $B(t)$, interpreted as the column vector associated with the $i$th control component $u_i$.
	
If ${\cal A} \cap {\cal B} \neq \emptyset$\,, then one has a {\em feasible} LQ optimal control problem.  The {\em feasibility problem} is posed as one of finding an element in ${\cal A} \cap {\cal B}$, namely:
\begin{equation}  \label{prob:feas}
	\mbox{Find\ \ } u \in {\cal A} \cap {\cal B}\,.
\end{equation}
If, however, ${\cal A} \cap {\cal B} = \emptyset$\,, then the problem is said to be {\em infeasible}.  The feasibility problem in~\eqref{prob:feas} has obviously no solution in this case, but in Section~\ref{sec:best_approx} we will pose the problem of finding (in some sense) a {\em best approximation solution}.

\section{Best Approximation Solution to the Infeasible Problem}
\label{sec:best_approx}

Consider the case when ${\cal A} \cap {\cal B} = \emptyset$.  
We define the {\em best approximation pair} as 
$(u_{\cal A}^*, u_{\cal B}^*) \in {\cal A}\times{\cal B}$ which minimizes the squared distance between the two sets.  Namely $(u_{\cal A}^*, u_{\cal B}^*)$ is in this case required to solve
\begin{equation}  \label{prob:min_dist}
	\min_{\substack{u_{\cal A}\in {\cal A} \\ u_{\cal B}\in {\cal B}}}\ \ \ds\frac{1}{2}\,\|u_{\cal A} - u_{\cal B}\|_{L^2}^2\,,
\end{equation}
where $\|\cdot\|_{L^2}$ is the ${L^2}$ norm. In other words, we want to minimize the ``gap'' between the two sets.  
Observe that $\cal{A}-\cal{B}$ is convex, closed (by, e.g., \cite[Proposition~3.42]{BC2017}),
and nonempty. Therefore, it
follows from \cite[Section~2]{BauBor94}
 and the fact that ${{\cal A}-{\cal B}}$ is closed
 that
\begin{equation}  \label{gap_vbest}
	u^*_{\cal A} - u^*_{\cal B}\,=P_{{{\cal A}-{\cal B}}}(0).
\end{equation}

We define the {\em gap} ({\em function}) {\em vector} (see \cite{BauMou2017})
\begin{equation}  \label{gap_v}
	v := u_{\cal A} - u_{\cal B},\,
 \qquad (u_{\cal A}, u_{\cal B}) \in {\cal A}\times{\cal B}.
\end{equation}
Using $u_{\cal A} = v + u_{\cal B}$ and the definitions of ${\cal A}$ and ${\cal B}$ in~\eqref{A}--\eqref{B}, the problem in~\eqref{prob:min_dist} can be rewritten in the format of a classical, or standard, optimal control problem as follows.
\[
\mbox{(Pf) }\left\{\begin{array}{rl}
\ds\min_{v(\cdot),u_{\cal B}(\cdot)} & \ \ \ds\frac{1}{2}\int_0^1 \|v(t)\|_2^2\, dt \\[2mm] 
\mbox{subject to} 
& \ \ \ds\dot{x}(t) = A(t)x(t) + \sum_{i=1}^m b_i(t)(v_i(t) + u_{{\cal B},i}(t))\,,\ \ \mbox{for a.e.\ } t\in[0,1]\,, \\[2mm]
& \ \ \varphi(x(0),x(1)) = 0\,, \\[2mm]
& \ \ \underline{a}_i(t) \le u_{{\cal B},i}(t) \le \overline{a}_i(t)\,,\ \ \mbox{for a.e.\ } t\in[0,1]\,,\ \ i = 1,\ldots,m\,,
\end{array} \right.
\]
where $\|\cdot\|_2$ is the Euclidean norm.  Problem~(Pf) is an optimal control problem with two control variable vectors, namely $v$ and $u_{\cal B}$, where $v(t) := (v_1(t),\ldots,v_m(t))\in\dR^m$ and\linebreak $u_{\cal B}(t) := (u_{{\cal B},1}(t),\ldots,u_{{\cal B},m}(t))\in\dR^m$.

\subsection{Properties of Problem~(Pf)}

Denote by $S_f$ the set of solutions of Problem~(Pf). Recall that, for any given set $C$ of a Hilbert space $H$, the {\em indicator function of $C$}, denoted by $\iota_C:H\to \mathbb{R}\cup \{+\infty\}$, is defined as $\iota_C(x)=0$ for every $x\in C$, and $\iota_C(x)=+\infty$ for every $x\not\in C$. We show in this section the main properties of Problem~(Pf). 

\begin{lemma}\label{lem:D1}
    The constraint set of Problem~{\em (Pf)} is convex and (strongly and weakly) closed (i.e., closed w.r.t. the norm topology and w.r.t. the weak topology in $L^2$).
\end{lemma}
\begin{proof}
The constraint set of Problem~(Pf) can be written as follows:
\begin{equation}\label{rem:pro}
    {\cal D}:=\{ (v,u_{\cal B})\in L^{2}([0,1];\dR^m)\times L^{2}([0,1];\dR^m)\::\:  v+u_{\cal B} \in {\cal A},\, u_{\cal B} \in {\cal B}\},
\end{equation}
where ${\cal A},\,{\cal B}$ are as in \eqref{A} and \eqref{B}, respectively. Consider the map
\[
\Psi:L^{2}([0,1];\dR^m)\times L^{2}([0,1];\dR^m)\to L^{2}([0,1];\dR^m)\times L^{2}([0,1];\dR^m),\] 
defined by $\Psi(z,x):=(z+x,x)$. The map $\Psi$ is a linear bijection which is continuous in $L^2$. The convexity of ${\cal D}$ now follows from the fact that $\Psi^{-1}$ is linear and ${\cal A}\times {\cal B}$ (being the product of an affine set and a box), is convex. Note also that ${\cal A}\times {\cal B}$ is closed because each factor is closed. Indeed, ${\cal A}$ is closed by \eqref{eq:A closed}.  The set ${\cal B}$ is closed in $L^{2}([0,1];\dR^m)$ because every sequence converging in $L^{2}([0,1];\dR^m)$ has a subsequence converging a.e. in $[0,1]$. The latter implies that every limit in the topology of $L^2$ must belong to ${\cal B}$. Altogether, the set ${\cal A}\times {\cal B}$ is closed in $L^{2}([0,1];\dR^m)$ and therefore ${\cal D}$ is closed because it is the preimage of a closed set by a continuous function. The fact that the closedness holds for both the strong and weak topology follows from convexity.
\end{proof}

\begin{lemma}\label{lem:D2}
 The solution set of Problem~{\em (Pf)} is not empty. Moreover, if $(v_1,u_1),\, (v_2,u_2)$ solve Problem~{\em (Pf)}, then $v_1=v_2$ (i.e., the coordinate $v$ of any solution to Problem~{\em (Pf)} is unique).
\end{lemma}
\begin{proof}
Note first that (Pf) can be equivalently written as having for objective function 
\[
h(v,u)= \frac{1}{2}\int_0^1 \|v(t)\|_2^2 +\iota_{\cal B}(u),
\]
where $\iota_{\cal B}$ is the indicator function of the set ${\cal B}$ and ${\cal B}$ is as in \eqref{B}. Now the second statement follows directly from the fact that $h$ is strongly convex in the variable $v$. We proceed next to prove the first statement. Since the functions $\underline{a},\, \overline{a}$ are continuous, the set ${\cal B}$ is bounded, and hence $h$ is coercive in both variables. Consider the set ${\cal D}$ as in the proof of Lemma \ref{lem:D1}. The coerciveness of $h$ allows us to find a closed ball $B[0,R]$ such that a solution of (Pf) (if any) must be in $D_0:={\cal D}\cap B[0,R]$. By Lemma \ref{lem:D1}, $D_0$ is convex and closed. It is also bounded because it is contained in the ball $B[0,R]$. By Bourbaki-Alaoglu and convexity, $D_0$ is weakly compact. Since the function $h_1(v):=\frac{1}{2}\int_0^1 \|v(t)\|_2^2$ is continuous and convex, it is weakly lower-semicontinuous. Recall that the set ${\cal B}$ is closed and convex, and hence weakly closed and convex, Therefore, the function $h_2:=\iota_{\cal B}$ is weakly lower-semicontinuous. Altogether, 
\( h= h_1+h_2\) is weakly lower-semicontinuous. We can now consider the problem 
\[
\mbox{(PD)} \min_{(z,x)\in D_0} h(z,x).
\]
By construction, the solution set of Problem~(PD) is $S_f$. Since $h$ is weakly lower-semicontinuous and $D_0$ is weakly compact, Problem~(PD) has a solution, and hence Problem~(Pf) has (the same)  solution(s).
\end{proof}

\subsection{\bf Maximum Principle for Problem~(Pf)}

In what follows we will derive the necessary conditions of optimality for Problem~(Pf), using the {\em maximum principle}.  Various forms of the maximum principle and their proofs can be found in a number of reference books---see, for example, \cite[Theorem~1]{PonBolGamMis1986}, \cite[Chapter~7]{Hestenes1966}, \cite[Theorem~6.4.1]{Vinter2000}, \cite[Theorem~6.37]{Mordukhovich2006}, and \cite[Theorem~22.2]{Clarke2013}.  We will state the maximum principle suitably utilizing these references for our setting and notation.  
	
First, define the {\em Hamiltonian function} $H:\dR^n \times \dR^m \times \dR^m \times \dR^n \times [0,1]\to \dR$ for Problem~(Pf) as
\begin{equation}  \label{Hamiltonian}
	H(x(t),v(t),u_{\cal B}(t),\lambda(t),t) := \frac{1}{2}\,\|v(t)\|_2^2 + \left\langle\lambda(t), A(t)\,x(t) + \sum_{i=1}^m b_i(t)\,(v_i(t) + u_{{\cal B},i}(t))\right\rangle,
\end{equation}
where $\lambda(t) := (\lambda_1(t),\ldots,\lambda_n(t))\in\dR^n$ is the {\em adjoint variable} (or {\em costate}) {\em vector} such that
\[ 
\dot{\lambda}(t) := -\frac{\partial H}{\partial x}(x(t), v(t), u_{\cal B}(t), \lambda(t), t)\,,
\] 
i.e.,
\begin{equation} \label{adjoint} 
	\dot{\lambda}(t) = -A^T(t)\,\lambda(t)\,,
\end{equation} 
where the {\em transversality conditions} involving $\lambda(0)$ and $\lambda(1)$ depend on the boundary condition $\varphi(x(t_0),x(t_f)) = 0$, but are not  expressed here.

\noindent	
{\bf Maximum Principle.}\ \  Suppose that the triplet
\[
(x,v,u_{\cal B})\in W^{1,\infty}([0,1];\dR^n) \times L^2([0,1];\dR^m) \times L^\infty([0,1];\dR^m)
\]
is optimal for Problem~(Pf).  Then there exists a continuous adjoint variable vector $\lambda\in W^{1,\infty}([0,1];\dR^n)$ as defined in~\eqref{adjoint}, such that $\lambda(t)\neq{\bf 0}$ for all $t\in[0,1]$, and that, for a.e.\ $t\in[0,1]$,
\begin{equation}  \label{opt_v_gen}
	\frac{\partial H}{\partial v_i}(x(t), v(t), u_{\cal B}(t), \lambda(t), t) = v_i(t) + b_i^T(t)\,\lambda(t) = 0\,,
\end{equation}
and
\begin{equation}  \label{opt_uB_gen}
	u_{{\cal B},i}(t) = \argmin_{w_i\in[\underline{a}_i(t), \overline{a}_i(t)]} H(x(t), v(t), w_i, \lambda(t), t) = \argmin_{w_i\in[\underline{a}_i(t), \overline{a}_i(t)]} b_i^T(t)\,\lambda(t)\, w_i\,,
\end{equation}
for $i = 1,\ldots,m$.  Condition~\eqref{opt_v_gen} can in turn be rewritten as
\begin{equation}  \label{opt_vi}
	v_i(t) = -b_i^T(t)\lambda(t)\,,
\end{equation}
for $i = 1,\ldots,m$, i.e.,
\begin{equation}  \label{opt_v}
	v(t) = -B^T(t)\lambda(t)\,,
\end{equation}
for all $t\in[0,1]$.  On the other hand, Condition~\eqref{opt_uB_gen} results in, also by incorporating~\eqref{opt_vi},
\begin{equation}  \label{opt_uB}
	u_{{\cal B},i}(t) = \left\{\begin{array}{rl}
	\overline{a}_i(t)\,, &\ \ \mbox{if\ \ } v_i(t) > 0\,, \\[1mm]
	\underline{a}_i(t)\,, &\ \ \mbox{if\ \ } v_i(t) < 0\,, \\[1mm]
	\mbox{undetermined}\,, &\ \ \mbox{if\ \ } v_i(t) = 0\,,
		\end{array} \right.
\end{equation}
for a.e.\ $t\in[0,1]$, $i = 1,\ldots,m$.  
	
	
The expression in \eqref{opt_uB} prompts two types of optimal control that are widely studied in the optimal control literature, as elaborated next.
	
\noindent 
{\bf Bang--Bang and Singular Types of Optimal Control.}\ \  If $v_i(t) \neq 0$ for a.e.\ $t\in[t',t'']\subset[0,1]$ with $t' < t''$, then the optimal control $u_{{\cal B},i}(t)$ in~\eqref{opt_uB} is referred to be of {\em bang--bang} type in the interval $[t',t'']$.  In this case, the optimal control might {\em switch} from $u_{{\cal B},i}(t) = \overline{a}_i(t)$ to $u_{{\cal B},i}(t) = \underline{a}_i(t)$, or vice versa, at some finitely many {\em switching times} in $[t',t'']$.  However, if $v_i(t) = 0$ for a.e.\ $t\in[s',s'']\subset[0,1]$, $s'<s''$,  then the optimal control is said to be of {\em singular} type in the interval $[s',s'']$.  Note that in general the optimal control might also switch from a bang-arc to a singular arc, and vice versa.
	
The optimality conditions we have just derived in~\eqref{opt_v}--\eqref{opt_uB} for Problem~(Pf) give rise to Theorem~\ref{thm:gap&bbang} stated further below.  If the dynamical control system is {\em controllable}, a definition of which is to be provided next, then the theorem eliminates the singularity for $u_i$, i.e., that the condition $v_i(t) = 0$ in~\eqref{opt_uB} can happen only at isolated time instants, and expresses the optimal $u_{{\cal B},i}(\cdot)$ as a control which is of bang--bang type.

Before stating Theorem~\ref{thm:gap&bbang} on the best approximation solution we first discuss the concept of controllability and some existing results.
 
\subsection{Controllability}
\label{sec:controllability}

The state equation, or the control system, 
\begin{equation}  \label{contr_sys}
	\dot{x}(t) = A(t)x(t) + B(t)u(t),
\end{equation}
is said to be {\em controllable} on a finite interval $[t_0,t_f]$ if given any initial state $x(t_0) = x_0$ there exists a continuous control $u(\cdot)$ such that the corresponding solution of~\eqref{contr_sys} satisfies $x(t_f) = 0$.
	
The solution of the (uncontrolled) system $\dot{x}(t) = A(t)x(t)$, with $x(0) = x_0$, is given by $x(t) = \Phi_A(t_0,t_f)\,x_0$.  Recall that $\Phi_A(t_0,t_f)$ is the {\em state transition matrix}, or the fundamental matrix, of the differential equation.
	
\begin{theorem}[Controllability via a Gramian Test Matrix~{\cite[Theorem~9.2]{Rugh1996}}]  \label{contr_Gram} \ \\
The system in~\eqref{contr_sys} is controllable on $[t_0,t_f]$ if and only if the $n\times n$ (Gramian) matrix
\[
	W(t_0,t_f) := \int_{t_0}^{t_f} \Phi_A(t_0,t_f)B(t)B^T(t)\Phi_A^T(t_0,t_f)\,dt\,,
\]
is invertible.
\end{theorem}
The matrix $W(t_0,t_f)$ defined above is called the {\em controllability Gramian}, and in general it is not easy to compute, making Theorem~\ref{contr_Gram} rather impractical. Hence, we present next a computable version of this result. Suppose that $A(\cdot)$ and $B(\cdot)$ are not only continuous but also ``sufficiently" smooth.  Let
\begin{subequations}
	\begin{eqnarray}
		&& K_0(t) := B(t)\,, \label{K0}\\
		&& K_j(t) := -A(t)K_{j-1}(t) + \dot{K}_{j-1}(t)\,,\ \ j = 1,2,\ldots\,. \label{Kj}
	\end{eqnarray}
\end{subequations}
With these definitions, a (much more easily) computable version of Theorem~\ref{contr_Gram} can be given as follows.
\begin{theorem}[Controllability via a More Easily Computable Test Matrix] \label{contr_matrix} \ \\
{\bf(\!\!{\cite[Theorem~9.4]{Rugh1996}}, \cite{SilMea1967})} Suppose $q$ is a positive integer such that, on $[t_0,t_f]$, $B(t)$ is $q$-times continuously differentiable, and $A(t)$ is $(q-1)$-times continuously differentiable.  Then the system in~\eqref{contr_sys} is controllable on $[t_0,t_f]$ if for some $t_c\in[t_0,t_f]$,
\begin{equation}  \label{contr_matrix_rank}
	\rank \left[K_0(t_c)\ |\ K_1(t_c)\ |\ \cdots\ |\ K_q(t_c)\right] = n\,,
\end{equation}
with $K_j(t_c)$, $j = 1,\ldots,q$, computed using~\eqref{K0}--\eqref{Kj}, is invertible.
\end{theorem}
Checking~\eqref{contr_matrix_rank} is in general far easier than checking the invertibility of $W(t_0,t_f)$.
	
\noindent 
{\bf Component-wise Controllability.}\ \  
We call the control system in~\eqref{contr_sys} {\em controllable w.r.t.\ $u_i$} on $[t_0,t_f]$ if given any initial state $x(t_0) = x_0$ there exists a continuous $i$th component $u_i(\cdot)$ of the control $u(\cdot)$ such that the corresponding solution of
\begin{equation}  \label{contr_sys_ui}
	\dot{x}(t) = A(t)x(t) + b_i(t)u_i(t)
\end{equation}
satisfies $x(t_f) = 0$.  Then clearly Theorems~\ref{contr_Gram} and \ref{contr_matrix}, with $B(t)$ replaced by $b_i(t)$, hold for the system in~\eqref{contr_sys_ui}, as the component-wise definition of controllability is stronger than that for the more general definition we gave originally.

\subsection{Best approximation solution}

Next, we provide in a theorem the best approximation solution in the set ${\cal B}$, in the case when the constraint sets ${\cal A}$ and ${\cal B}$ are disjoint. 
\begin{theorem}[Gap Vector and the Best Approximation Control in \boldmath{${\cal B}$}]  \label{thm:gap&bbang}\ \ 
With the notation of Problem $(P_f)$, assume that ${\cal A} \cap {\cal B} = \emptyset$.  Then the optimal gap vector is given by $v(t) = -B^T(t)\lambda(t)$, for all $t\in[0,1]$, where $\lambda(\cdot)$ solves~\eqref{adjoint}.  Moreover, suppose that $A(\cdot)$ and $B(\cdot)$ are sufficiently smooth and that the control system~\eqref{contr_sys} is controllable w.r.t.\ $u_i$ on any $[t',t'']\subset[0,1]$, $t'<t''$, for some $i = 1,\ldots,m$.  Then, for a.e.\ $t\in[0,1]$,
\begin{equation}  \label{uB}
	u_{{\cal B},i}(t) = \left\{\begin{array}{rl}
	\overline{a}_i(t)\,, &\ \ \mbox{if\ \ } v_i(t) \ge 0\,, \\[1mm]
	\underline{a}_i(t)\,, &\ \ \mbox{if\ \ } v_i(t) < 0\,.
	\end{array} \right.
\end{equation} 
In other words, such $u_{{\cal B},i}$ is of bang--bang type.
\end{theorem}
\begin{proof}
With ${\cal A} \cap {\cal B} = \emptyset$, the optimal gap vector $v(\cdot)$ is simply given in~\eqref{opt_v} as part of the solution to (the always feasible) Problem~(Pf).  This establishes the first assertion. We proceed to prove the last assertion. For contradiction purposes, suppose that, for the index $i$ as in the hypothesis (i.e., verifying the controllability assumption), the solution is not (only) bang--bang. Using \eqref{opt_uB}, this means that $v_i(t) = 0$ for a.e.\ $t\in[t',t'']\subset[0,1]$, $t'<t''$.  Then the $k$th derivative of $v(\cdot)$ is also zero over this nontrivial interval. Namely, $v_i^{(k)}(t) = 0$, for a.e.\ $t\in(t',t'')$, and all $1\le k\le (n-1)$. Note that using~\eqref{adjoint} and \eqref{opt_vi} one has, for a.e.\ $t\in(t',t'')$,
\begin{subequations}
	\begin{eqnarray}
		&& v_i(t) = -\lambda^T(t)\,b_i(t) = 0\,,  \label{vi_0} \\
		&& \dot{v}_i(t) = -\dot{\lambda}^T(t)\,b_i(t) -\lambda^T(t)\,\dot{b}_i(t) = \lambda^T(t)\left(A(t)\,b_i(t) - \dot{b}_i(t)\right) = 0\,, \\
		&& \ddot{v}_i(t) = \dot{\lambda}^T(t)\left(A(t)\,b_i(t) - \dot{b}_i(t)\right) + \lambda^T(t)\left(\dot{A}(t)\,b_i(t) + A(t)\,\dot{b}_i(t) - \ddot{b}_i(t)\right) \nonumber \\
		&& \hspace*{8.5mm} = \lambda^T(t)\left(-A^2(t)\,b_i(t) + 2\,A(t)\,\dot{b}_i(t) + \dot{A}(t)\,b_i(t) - \ddot{b}_i(t) \right) = 0\,,  \label{vi_3} \\
		&& \mbox{and so on.} \nonumber
	\end{eqnarray}
\end{subequations}
Let $p_{i,0}(t) := b_i(t)$.  Equations~\eqref{vi_0}--\eqref{vi_3} can be rewritten as
\begin{subequations}
	\begin{eqnarray}
		v_i(t) &=& -\lambda^T(t)\,p_{i,0}(t) = 0\,,  \label{vi_1} \\
		v_i^{(k)}(t) &=& -\lambda^T(t)\,p_{i,k}(t) = 0\,,\ \ k = 1,2,\ldots\,,   \label{vi_k}
	\end{eqnarray}
\end{subequations}
where
\begin{subequations}
	\begin{eqnarray}
		&& p_{i,0}(t) := b_i(t)\,, \label{pi_0} \\
		&& p_{i,k}(t) := -A(t)\,p_{i,k-1}(t) + \dot{p}_{i,k-1}(t)\,,\ \ k = 1,2,\ldots\,. \label{pi_k}  \label{pi_k}
	\end{eqnarray}
\end{subequations}
Note that $p_{i,k}(t)$, $k = 1,2,\ldots$, are the same as $K_j(t)$, $k = 1,2,\ldots$, in~\eqref{K0}--\eqref{Kj}, but with $B(t)$ replaced by $b_i(t)$.  From \eqref{vi_1}--\eqref{vi_k}, one gets
\begin{equation}  \label{Qci}
	-\lambda^T(t)\,Q_c^i(t) = 0\,,
\end{equation}
where
\begin{equation}  \label{Qci_def}
	Q_c^i(t) := [\ p_{i,0}(t)\ |\  p_{i,1}(t)\ |\  \ldots\ |\ p_{i,n-1}(t)\ ]\,. 
\end{equation}
Suppose that the control system is controllable w.r.t.\ $u_i$ on $[t',t'']$.  Then by Theorem~\ref{contr_matrix} there exists some $t_c\in[t',t'']$ such that $\rank Q_c^i(t_c) = n$.  This implies from~\eqref{Qci} that $\lambda(t_c) = 0$, and that in turn implies by the ODE in~\eqref{adjoint} that $\lambda(t) = 0$ for all $t\in[0,1]$, which is not allowed by the maximum principle.  Therefore one cannot have that $v_i(t) = 0$ for a.e.\ $[t',t'']\subset[0,1]$, as a result giving rise to~\eqref{uB}.
\end{proof}
	
\begin{remark}[The Best Approximation Control in \boldmath{${\cal A}$}] \rm
Consider the expressions for the $i$th component of the optimal gap vector $v(\cdot)$ and the $i$th component of the best approximation control $u_{{\cal B}}(\cdot)$, given as in \eqref{gap_v} and \eqref{uB}, respectively.  One can then simply express the $i$th component of the best approximation control in the affine set ${\cal A}$ as
\begin{equation}  \label{uA}
	u_{{\cal A},i}(t) = \left\{\begin{array}{rl}
	\overline{a}_i(t) + v_i(t)\,, &\ \ \mbox{if\ \ } v_i(t) \ge 0\,, \\[1mm]
	\underline{a}_i(t) + v_i(t)\,, &\ \ \mbox{if\ \ } v_i(t) < 0\,,
	\end{array} \right.
\end{equation}
for a.e.\ $t\in[0,1]$.  We observe that while $v_i(\cdot)$ as given in~\eqref{opt_vi} is continuous, $u_{{\cal A},i}(\cdot)$ is piecewise continuous.
\proofbox
\end{remark}

\begin{remark}[Time-invariant Systems] \em
Suppose that the control system in~\eqref{contr_sys} is time-invariant; namely that $A(t) = A$ and $B(t) = B$, $A$ and $B$ constant matrices, for all $t\in[0,1]$.  This is a widely encountered case in control theory although the time-varying case is more general.  We note that, in~\eqref{pi_0}--\eqref{pi_k}, $\dot{p}_{i,k-1} = 0$ and so we can write
\[
	p_{i,k}(t) := (-1)^{k+1}\,A^k\,b_i\,,\ \ k = 0,1,\ldots,n-1\,.
 \]
Since $p_{i,k}(\cdot)$ is constant, write $Q_c^i := Q_c^i(t_c)$.  In turn the rank condition $\rank Q_c^i = n$ can explicitly be stated as
\begin{equation}  \label{Kalman}
	\rank\, [\ b_i\ |\  A\,b_i\ |\  \ldots\ |\ A^{n-1}\,b_i\ ] = n\,.
\end{equation}
The condition in~\eqref{Kalman} for time-invariant control systems is referred to as the {\em Kalman controllability rank condition}~\cite{Rugh1996} in control theory.  In conclusion, for invariant systems, if the rank condition in \eqref{Kalman} holds then the control component $u_{{\cal B},i}$ for the infeasible optimal control problem is of bang--bang type as given in \eqref{uB}.
\proofbox
\end{remark}

\section{Critical Feasibility}
\label{sec:crit_feas}
	
Suppose that $\overline{a}_i(t) = a > 0$ and $\underline{a}_i(t) = -a$ for all $t\in[0,1]$ and $i = 1,\ldots,m$.  Since it is assumed that ${\cal A} \neq \emptyset$, if $a = \infty$  or large enough, the optimal control problem given in~\eqref{obj_fun}--\eqref{bounds} is feasible, i.e., ${\cal A} \cap {\cal B} \neq \emptyset$.  By the same token, if $a$ is small enough, the problem is infeasible for some specified initial and terminal end states, i.e., ${\cal A} \cap {\cal B} = \emptyset$.  In fact, from the geometry of the sets ${\cal A}$ and ${\cal B}_a$, where ${\cal B}_a$ indicates the explicit dependence of ${\cal B}$ on $a$, there exists a {\em critical bound} $a_c$ such that for all $a < a_c$, ${\cal A} \cap {\cal B}_a = \emptyset$, since ${\cal B}_a$ is strictly contained by (or strictly smaller than) ${\cal B}_{a_c}$.  By this definition, when $a=a_c$ we say that the problem is {\em critically feasible}.
	
We are interested in knowing when a problem becomes critically feasible.  In other words, we want to find the smallest value $a_c$ of $a$ for which the problem is feasible.  We can pose this question as a new (parametric) optimal control problem, where the {\em parameter} $a$ is to be minimized subject to the constraint sets ${\cal A}$ and ${\cal B}_a$:
\[
\mbox{(Pcf) }\left\{\begin{array}{rl}
	\ds\min_{u(\cdot),\,a} & \ \ a \\[1mm] 
	\mbox{subject to} & \ \ \ds\dot{x}(t) = A(t)x(t) + \sum_{i=1}^m b_i(t)u_i(t)\,,\ \ \mbox{for a.e.\ } t\in[0,1]\,, \\[2mm]
	& \ \ \varphi(x(0),x(1)) = 0\,, \\[2mm]
	& \ \ |u_i(t)| \le a\,,\ \ \mbox{for a.e.\ } t\in[0,1]\,,\ \ i = 1,\ldots,m\,,
\end{array} \right.
\]
the optimal value of which will be $a_c$.  

\begin{remark} \rm
We observe that $|u_i(t)| \le a$,  $i = 1,\ldots,m$, can be written as $\|u(t)\|_{\infty}\le a$, where $\|\cdot\|_{\infty}$ is the $\ell_{\infty}$-norm in $\dR^m$.  By also observing that the problem of ``{\em minimizing the value of the variable $a$ subject to $\|u(t)\|_{\infty}\le a$, for a.e.\ $t\in [0,1]$},'' is equivalent to ``{\em minimizing the $L^\infty$-norm of $u$},'' Problem~(Pcf) can be re-written as follows.
\[
\mbox{(Pcf1) }\left\{\begin{array}{rl}
	\ds\min & \ \ \|u\|_{L^\infty}  \\[1mm] 
	\mbox{subject to} & \ \ \ds\dot{x}(t) = A(t)x(t) + \sum_{i=1}^m b_i(t)u_i(t)\,,\ \ \mbox{for a.e.\ } t\in[0,1]\,, \\[2mm]
	& \ \ \varphi(x(0),x(1)) = 0\,.
\end{array} \right.
\]
It is interesting to note that Problem~(Pf1) is a generalized form of the problem studied in~\cite{KayNoa2013}.  In what follows we will use the procedure in~\cite{KayNoa2013}.
\proofbox
\end{remark}

Before we apply the maximum principle, it is convenient to re-write Problem~(Pcf) as an optimal control problem in standard (or classical) form.  First, we define a new state variable $y(t) := a$ and a new control variable
\begin{equation}  \label{def:v}
	w(t) := u(t)/a\,.  
\end{equation}  
Problem~(Pcf) can then be re-cast using these new variables as
\[
\mbox{(Pcf2) }\left\{\begin{array}{rl}
	\ds\min & \ \ y(1)  \\[1mm] 
	\mbox{subject to} & \ \ \ds\dot{x}(t) = A(t)x(t) + y(t)\sum_{i=1}^m b_i(t)\,w_i(t)\,,\ \ \mbox{for a.e.\ } t\in[0,1]\,, \\[2mm]
	& \ \ \varphi(x(0),x(1)) = 0\,, \\[2mm]
	& \ \ \dot{y}(t) = 0\,,  \ \ \ \ |w_i(t)|\le 1\,,\ \ \mbox{for a.e.}\ t\in [0,1]\,,\ \ i = 1,\ldots,m\,.
\end{array} \right.
\]
The Hamiltonian function $H:\dR^n \times \dR \times \dR^m \times \dR^n \times \dR \times [0,1] \to \dR$ for the critical feasibility problem (Pcf2) can be written as
\begin{equation}  \label{Hamiltonian_Pcf2}
	H(x(t),y(t),w(t),\lambda(t),\mu(t),t) := \left\langle\lambda(t), A(t)\,x(t) + y(t) \sum_{i=1}^m b_i(t)\,w_i(t)\right\rangle + \mu(t)\cdot 0\,,
\end{equation}
where $\lambda(t) := (\lambda_1(t),\ldots,\lambda_n(t))\in\dR^n$ is the {\em adjoint variable} solving~\eqref{adjoint}, and $\mu(t)$ is an additional adjoint variable such that
\[ 
\dot{\mu}(t) := -\frac{\partial H}{\partial y}(x(t), y(t), w(t), \lambda(t), \mu(t), t)\,,\]
so that
\begin{equation} \label{adjoint_mu} 
	\dot{\mu}(t) = -w^T(t)\,B^T(t)\,\lambda(t)\,,\quad \mu(0) = 0\,,\ \ \mu(1) = 1\,.
\end{equation} 
	
\subsection{\bf Maximum Principle for Problem~(Pcf2)}

Suppose that the triplet
\[
(x,y,w)\in W^{1,\infty}([0,1];\dR^n) \times W^{1,\infty}([0,1];\dR) \times L^\infty([0,1];\dR^m)
\]
is optimal for Problem~(Pcf2).  Then there exists a continuous adjoint variable vector $\lambda\in W^{1,\infty}([0,1];\dR^n)$ as defined in~\eqref{adjoint} and an additional continuous adjoint variable\linebreak $\mu\in W^{1,\infty}([0,1];\dR)$ as defined in~\eqref{adjoint_mu} such that $(\lambda(t),\mu(t))\neq{\bf 0}$ for all $t\in[0,1]$, and that, for a.e.\ $t\in[0,1]$,
\begin{equation}  \label{opt_w}
	w(t) = \argmin_{r_i\in[-1,1]} H(x(t), y(t), r_i, \lambda(t), \mu(t), t) = \argmin_{r_i\in[-1,1]} y(t)\,b_i^T(t)\lambda(t) r_i\,,
\end{equation}
for $i = 1,\ldots,m$.  Condition~\eqref{opt_w} results in
\begin{equation}  \label{opt_wi}
	w_i(t) = \left\{\begin{array}{rl}
	1\,, &\ \ \mbox{if\ \ } b_i^T(t)\lambda(t) < 0\,, \\[1mm]
	-1\,, &\ \ \mbox{if\ \ } b_i^T(t)\lambda(t) > 0\,, \\[1mm]
	\mbox{undetermined}\,, &\ \ \mbox{if\ \ } b_i^T(t)\lambda(t) = 0\,,
    \end{array} \right.
\end{equation}
for a.e.\ $t\in[0,1]$, $i = 1,\ldots,m$.

We show next that the solution of (Pcf) is bang--bang. Namely, there is no nontrivial subinterval of $[0,1]$ where $b_i^T(t)\lambda(t)$ vanishes almost everywhere.

\subsection{Solution to the critically feasible problem}

\begin{theorem}[Critically Feasible Control] 
\label{theo:crit}
Suppose that the system and control matrices $A(\cdot)$ and $B(\cdot)$ are sufficiently smooth.  Assume that, for some $i = 1,\ldots,m$, the control system~\eqref{contr_sys} is controllable w.r.t.\ $u_i$ on any $[s',s'']\subset[0,1]$, $s'<s''$. Then, the $i$th component $u_i(\cdot)$ of the critically feasible control for the optimal control problem in~\eqref{obj_fun}--\eqref{bounds}, with $\overline{a}_i(t) = a_c$ and $\underline{a}_i(t) = -a_c$, is given as
\begin{equation}  \label{u_crit_feas}
	u_i(t) = \left\{\begin{array}{rl}
	a_c\,, &\ \ \mbox{if\ \ } b_i^T(t)\lambda(t) \le 0\,, \\[1mm]
	-a_c\,, &\ \ \mbox{if\ \ } b_i^T(t)\lambda(t) > 0\,,
	\end{array} \right.
\end{equation}
for a.e.\ $t\in[0,1]$, where $\lambda(\cdot)$ solves~\eqref{adjoint}.  In other words, such $u_i$ is of bang--bang type.

\end{theorem}
\begin{proof}
Suppose that the control system~\eqref{contr_sys} is controllable w.r.t.\ $u_i$ on any $[s',s'']\subset[0,1]$, $s'<s''$.  For contradiction purposes, suppose that $w_i$ in~\eqref{opt_wi} is singular, i.e., $\lambda^T(t)\,b_i(t) = 0$ for all $t\in[t',t'']\subset[0,1]$, with $t' < t''$.  Then, as in the proof of Theorem~\ref{thm:gap&bbang}, the consecutive time-derivatives of $\lambda^T(t)\,b_i(t)$ will also equal to zero for all $t\in[t',t'']$.  Defining $p_{i,0}(t) := b_i(t)$ and $p_{i,k}(t) := -A(t)\,p_{i,k-1}(t) + \dot{p}_{i,k-1}(t)$ as before, for $k = 1,2,\ldots\,n-1$, one similarly gets~\eqref{Qci} (with a sign difference) and \eqref{Qci_def}.  Then by means of the same arguments using the controllability of~\eqref{contr_sys}, as in the proof of Theorem~\ref{thm:gap&bbang}, $\lambda(t) = 0$ for all $t\in[0,1]$.  Then the differential equation in~\eqref{adjoint_mu} with the initial condition $\mu(0) = 0$ in~\eqref{adjoint_mu} yields $\mu(t) = 0$ for all $t\in[0,1]$, contradicting the terminal condition $\mu(1) = 1$ in~\eqref{adjoint_mu}, and thus furnishing the theorem. \end{proof}
	
\begin{remark} \rm
We note that, in the critically feasible case, $u_{\cal A} = u_{\cal B} = u_i$ and so $v_i = 0$, for all $i = 1,\ldots,m$, and thus $v_i$ does not serve as the switching function for $u_i$.
\proofbox
\end{remark}

\subsection{A double integrator problem}
\label{sec:DI}

References~\cite{BauBurKay2019, BurCalKayMou2023} studied  applications of splitting and projection methods to the feasible problem of finding the so-called {\em minimum-energy control of the double integrator},  the problem stated as
	\[
	\mbox{(PDI) }\left\{\begin{array}{rl}
		\ds\min & \ \ \ds\frac{1}{2}\int_0^1 u^2(t)\,dt =  \frac{1}{2}\,\|u\|_{L^2}^2 \\[5mm] 
		\mbox{subject to} & \ \ \dot{x}_1(t) = x_2(t)\,,\ \ x_1(0) = s_0\,,\ \ x_1(1) = s_f\,, \\[2mm]
		& \ \ \dot{x}_2(t) = u(t)\,,\ \ \ \,x_2(0) = v_0\,,\ \ x_2(1) = v_f\,, \\[2mm]
		& \ \ |u(t)|\le a\,,\ \ \mbox{for all}\ \ t\in [0,1]\,.
	\end{array} \right.
	\]	
Although Problem~(PDI) constitutes a relatively simple instance of an optimal control problem, a solution to it can only be found numerically.  This is the first reason why we find it interesting.  Secondly, (PDI) acts as a building block in, for example, the problem of finding cubic spline interpolants with constrained acceleration, an active area of research in numerical analysis and approximation theory.  A much wider range of optimal control problems involving the double integrator have been studied in the relatively recent book~\cite{Locatelli2017}, however it does not include Problem~(PDI).  Problem~(PDI) is simple and yet rich enough to study when introducing and illustrating many basic and new concepts or when testing new numerical approaches in optimal control---see, in addition to~\cite{BauBurKay2019, BurCalKayMou2023, Locatelli2017}, also~\cite{BurKayLiu2023, Kaya2020}.
	
With a large enough $a$ (so that the constraint $|u(t)|\le a$ never becomes active for any $t\in[0,1]$), Problem~(PDI) can be solved analytically to find a cubic curve $x_1(t)$, satisfying the initial point and velocity $s_0$ and $v_0$, and the terminal point and velocity $s_f$ and $v_f$, respectively -- see~\cite{BauBurKay2019} for the working of such an unconstrained solution.  A small enough $a$, on the other hand, restricts the values the function $u$ can take and thus rules out finding an analytical solution and necessitates the use of a numerical procedure for finding an approximate solution.  This altogether furnishes a testimony to the practical significance of such a simple looking problem like~(PDI). \\

\noindent
{\bf Critical Bound \boldmath{$a_c$}\,: Going From Feasible to Infeasible.}\ \ 
For the numerical experiments in~\cite{BauBurKay2019}, the special case when $s_0 = s_f = v_f = 0$ and $v_0 = 1$ was considered.  The feasible optimal control for this instance of Problem~(PDI) is 
\begin{equation}  \label{optcont_spec}
	u(t) = \left\{\begin{array}{rl}
	-a\,, &\ \ \mbox{if\ \ } 0\le t < t_1\,, \\[1mm]
	\ds\frac{2a}{t_2 - t_1}\,(t - t_1) - a\,, &\ \ \mbox{if\ \ } t_1\le t < t_2\,, \\[1mm]
	a\,, &\ \ \mbox{if\ \ } t_2\le t \le 1\,,
	\end{array} \right.
\end{equation}
where $0 \le t_1 < t_2 \le 1$ are the so-called {\em junction times}. 
When the value of $a$ is too small, problem (PDI) becomes inconsistent. Namely, there exists a critical value $a_c>0$ such that, when $a < a_c$ Problem~(PDI) is infeasible.
Thus, the control constraint will be active when $a\in [a_c, 4)$. The latter is the consistent, or the feasible, case, for which the control solution $u$ is still active.    If $a=4$ or larger, then $t_1 = 0$ and $t_2 = 1$ and the solution is the same as that of the case when $u$ is unconstrained. In other words, when $a\ge 4$ the bound constraint on $u$ becomes superfluous. When $a\in (a_c,4)$, the solution $u$ of problem (P) given in \eqref{optcont_spec} is continuous over the time horizon $[0,1]$. On the other hand, when $a=a_c$, as has been stated in Theorem~\ref{theo:crit} that the control solution has to be of bang--bang type, or discontinuous. In Remark~2.1 of~\cite{BauBurKay2019}, it is observed that
\begin{equation}  \label{observation1}
	2.414 < a_c < 2.415\,,
\end{equation}
based on the numerical experiments conducted, without elaborating further.  It is also reported in the same remark that when $a = a_c$ the unique feasible solution appears to be bang--bang, i.e., in particular, $u(t)$ switches once from $-a_c$ to $a_c$ at  the {\em switching time}
\begin{equation}  \label{observation2}
	t = t_c \approx 0.71\,,
\end{equation}
confirming our statement above that the optimal control $u$ in this case is discontinuous. \\
	
\noindent
{\bf How to Find the Solution for \boldmath{$a_c$} and \boldmath{$t_c$}.}\ \  
The {\em Hamiltonian function} $H:\dR^3 \times \dR \times \dR^3 \to \dR$ for Problem~(Pcf2) emanating from Problem~(PDI) is
\[
	H((x_1(t),x_2(t),y(t)),w(t),(\lambda_1(t),\lambda_2(t),\mu(t))) := \lambda_1(t)\,x_2(t) + \lambda_2(t)\,y(t)\,w(t) + \mu(t)\cdot 0\,,
\]
where $(\lambda_1(t),\lambda_2(t),\mu(t))\in\dR^3$ is the {\em adjoint variable} (or {\em costate}) {\em vector} such that
\[
	\dot{\lambda}_1(t) = -\frac{\partial H}{\partial x_1} = 0\,,\quad
	\dot{\lambda}_2(t) = -\frac{\partial H}{\partial x_2} = -\lambda_1(t)\,,
 \quad\mbox{and}
 \quad
	\dot{\mu} = -\frac{\partial H}{\partial y} = -\lambda_2(t)\,w(t)\,,
\] 
with the transversality conditions
\begin{equation} \label{transversality1}
	\mu(0) = 0\quad\mbox{and}\quad \mu(1) = 1\,.
\end{equation}
This leads to the solutions
\begin{equation} \label{adjoint_sol} 
	\lambda_1(t) = c_1\,,\quad \lambda_2(t) = -c_1\,t - c_2\,,
\end{equation} 
where $c_1$ and $c_2$ are unknown real constants.  
 	
The following is a straightforward corollary to Theorem~\ref{theo:crit} for the double integrator problem.	
\begin{corollary}[Critically Feasible Control] \label{cor:crit_DI}
The critically feasible optimal control $u_c$ for the double integrator problem is of bang--bang type with at most one switching; namely
\begin{equation}  \label{u_crit}
	u_c(t) = \left\{\begin{array}{rl}
	a_c\,, &\ \ \mbox{if\ \ } 0\le t < t_c\,, \\[1mm]
	-a_c\,, &\ \ \mbox{if\ \ } t_c\le t \le 1\,,
    \end{array} \right.
\end{equation}
where $t_c$ is the switching time.
\end{corollary}
\begin{proof}
The lemma follows from the expression in~\eqref{u_crit_feas} in Theorem~\ref{theo:crit} and the linearity of $\lambda_2$ in~\eqref{adjoint_sol} (which implies that  $\lambda_2$ can change sign at most once).
\end{proof}

A similar line of proof with $a < a_c$ (the infeasible case) results in the following corollary to Theorem~\ref{thm:gap&bbang}.
\begin{corollary}[Best Approximation Control in \boldmath{${\cal B}$}] \label{cor: infeas_DI}
The best approximation optimal control $u_{\cal B}$ for the double integrator problem is of bang--bang type with at most one switching; namely
\begin{equation}  \label{u_B}
	u_{\cal B}(t) = \left\{\begin{array}{rl}
	a\,, &\ \ \mbox{if\ \ } 0\le t < t_s\,, \\[1mm]
	-a\,, &\ \ \mbox{if\ \ } t_s\le t \le 1\,,
    \end{array} \right.
\end{equation}
where $t_s$ is the switching time.
\end{corollary}
	
The theorem we present below provides the full analytical solution to Problem~(PDI) when $a = a_c$.
	
\begin{theorem}[Full Critically Feasible Solution to Problem~(PDI)] \label{control_crit} \ 
\begin{enumerate}
\item[(a)]  If $s_f - s_0 \neq (v_0 + v_f)/2$, then the critical control is given by
\begin{equation}  \label{u_crit1}
	u_c(t) = \left\{\begin{array}{rl}
	r\,, &\ \ \mbox{if\ \ } 0\le t < t_c\,, \\[1mm]
	-r\,, &\ \ \mbox{if\ \ } t_c\le t \le 1\,,
	\end{array} \right.
\end{equation}
and
\begin{equation}  \label{a_crit1}
	a_c = |r|\,,
\end{equation}
with $r$ and the switching time $t_c$ given in the following two cases.
\begin{enumerate}
\item[(i)] $v_0 \neq v_f$:
\begin{equation}  \label{r1}
	r = \frac{v_f - v_0}{2\,t_c - 1}\,
\end{equation}
and $t_c$ solves the quadratic equation
\begin{equation}  \label{t_crit}
	(v_f - v_0)\,t_c^2 + 2\,(s_f - s_0 - v_f)\, t_c + \frac{1}{2}\,(v_0 + v_f) - (s_f - s_0) = 0\,.
\end{equation}
\item[(ii)] $v_0 = v_f$:
\begin{equation}  \label{r2}
	r = 4\,(s_f - s_0 - v_0)\quad\mbox{and}\quad t_c = \frac{1}{2}\,.
\end{equation}
\end{enumerate}
			
\item[(b)]  If $s_f - s_0 = (v_0 + v_f)/2$, then $t_c = 0$ or $1$\,.  Furthermore, the critical control is given by
\begin{equation}  \label{u_crit2}
	u_c(t) = \left\{\begin{array}{rl}
	v_f - v_0\,,\ \ \mbox{if\ \ } t_c = 1\,, \\[1mm]
	v_0 - v_f\,,\ \ \mbox{if\ \ } t_c = 0\,,
\end{array} \right.
\end{equation}
for all $t\in [0,1]$, and so
\begin{equation}  \label{a_crit2}
	a_c = |v_f - v_0|\,.
\end{equation}
\end{enumerate}
\end{theorem}
\begin{proof}
Recall from \eqref{adjoint_sol} that $\lambda_2(t) = -c_1\,t - c_2$, for all $t\in[0,1]$.  Also note that $c_1$ and $c_2$ cannot both be zero, otherwise, with $\mu(0) = 0$ in~\eqref{transversality1} and continuity of $\mu$, it leads to $\mu(t) = 0$ for all $t\in[0,1]$, a contradiction with the fact that $\mu(1) = 1$ in~\eqref{transversality1}. Therefore in the rest of the proof we examine three cases: \\[-3mm]
		
\centerline{(I) $c_1 \neq 0$ and $c_2 = 0$;\ \ \ (II) $c_1 = 0$ and $c_2 \neq 0$;\ \ \ and\ \ \ (III) $c_1 \neq 0$ and $c_2 \neq 0$.}
		
\noindent
Case~(I): Suppose that $c_2 = 0$.  Then $\lambda_2(t) = -c_1\,t$, $c_1 \neq 0$, and by~\eqref{u_crit} $u(t) = \sgn(c_1)\,a_c$, for all $t\in[0,1]$ (no switching).  By solving the state equations with this $u(t)$ substituted, one gets $x_2(t) = \sgn(c_1)\,a_c\,t + v_0$ and $x_1(t) = \sgn(c_1)\,a_c\,t^2/2 + v_0\,t + s_0$, and subsequently $x_2(1) = v_f = \sgn(c_1)\,a_c + v_0$ and $x_1(1) = s_f = \sgn(c_1)\,a_c/2 + v_0 + s_0$.  Now, from these solutions, $\sgn(c_1)\,a_c = v_f - v_0$, and thus $s_f = (v_f - v_0)/2 + v_0 + s_0$, resulting in $s_f - s_0 = (v_0 + v_f)/2$, which is nothing but the case in part~(b) of the theorem. Finally one gets $u(t) = \sgn(c_1)\,a_c = v_f - v_0$ and thus $a_c = |v_f - v_0|$ as required by \eqref{u_crit2} and~\eqref{a_crit2}.
		
\noindent
Case~(II): Suppose that $c_1 = 0$. Then $\lambda_2(t) = -c_2$, $c_2 \neq 0$, and by~\eqref{u_crit} $u(t) = \sgn(c_2)\,a_c$, for all $t\in[0,1]$ (no switching).  The rest of the arguments follows similarly to the case when $c_2 = 0$ above simply by replacing $c_1$ by $c_2$ in the expressions.  This case also corresponds to and proves part~(b) of the theorem.
		
\noindent
Case~(III): Finally suppose that $\lambda_2(t) = -c_1\,t - c_2$ with both $c_1 \neq 0$ and $c_2 \neq 0$.  Then by Corollary~\ref{cor:crit_DI}, observing that $\lambda_2(0) = -c_2$,
\begin{equation}  \label{u_crit3}
	u_c(t) = \left\{\begin{array}{rl}
	r\,, &\ \ \mbox{if\ \ } 0\le t < t_c\,, \\[1mm]
	-r\,, &\ \ \mbox{if\ \ } t_c\le t \le 1\,,
	\end{array} \right.
\end{equation}
where
\[
r = \sgn(c_2)\,a_c > 0\,,
\]
verifying~\eqref{u_crit1}.  Next substitute $u(t) = u_c(t)$ into the differential equations in Problem~(Pc).  The respective solutions of $\dot{x}_2(t) = r$ with $x_2(0) = v_0$, and $\dot{x}_1(t) = x_2(t)$ with $x_1(0) = s_0$, for $0\le t < t_c$, are simply
\[
x_2(t) = r\,t + v_0\quad\mbox{and}\quad x_1(t) = \frac{1}{2}\,r\,t^2 + v_0\,t + s_0\,.
\]
Furthermore the respective solutions of $\dot{x}_2(t) = -r$ with $x_2(1) = v_f$, and $\dot{x}_1(t) = x_2(t)$ with $x_1(1) = s_f$, for $t_c\le t < 1$, can be obtained as
\[
x_2(t) = r\,(1-t) + v_f\quad\mbox{and}\quad x_1(t) = -\frac{1}{2}\,r\,(1-t)^2 - v_f\,(1-t) + s_f\,.
\]
Since $x_i$ are continuous, $\lim_{t\to t_c^-} x_i(t) = \lim_{t\to t_c^+} x_i(t)$, $i = 1,2$.  In other words,
\begin{equation}  \label{eqn:x2}
	r\,t_c + v_0 = r\,(1-t_c) + v_f\,,
\end{equation}
\begin{equation}  \label{eqn:x1}
	\frac{1}{2}\,r\,t_c^2 + v_0\,t_c + s_0 = -\frac{1}{2}\,r\,(1-t_c)^2 - v_f\,(1-t_c) + s_f\,.
\end{equation}
Case~(III)(ii): Suppose that $v_0 = v_f$.  Then, since $r>0$, $t_c = 1/2$ is the unique solution.  Substitution of $t_c = 1/2$ and $v_0 = v_f$ into \eqref{eqn:x1} and re-arrangements yield $r = 4\,(s_f - s_0 - v_0)$, or $a_c = 4\,|s_f - s_0 - v_0|$, verifying~\eqref{r2}.
		
\noindent
Case~(III)(i): 
Suppose that $v_0 \neq v_f$. Then, Equation~\eqref{eqn:x2} results in
\begin{equation}
	r = \frac{v_f - v_0}{2\,t_c - 1}\,,
\end{equation}
verifying \eqref{r1}.  After algebraic manipulations and re-arranging, \eqref{eqn:x1} can be rewritten as
\[
r\,t_c^2 - (v_f - v_0 + r)\,t_c + \frac{1}{2}\,r + v_f - (s_f - s_0) = 0\,.
\]
Substituting the expression for $r$ in \eqref{r1} into the above equation and multiplying both sides by $(2\,t_c - 1)$ give
\[
(v_f - v_0)\,t_c^2 - (v_f - v_0)\,t_c\,(2\,t_c - 1) - (v_f - v_0)\,t_c + \frac{1}{2}\,(v_f - v_0) + [v_f - (s_f - s_0)]\,(2\,t_c - 1) = 0\,.
\]
Further algebraic manipulations reduce the above equation to~\eqref{t_crit}, as required.  The proof is complete.
\end{proof}
	
\begin{remark}  \label{control_crit_ex} \rm
Suppose that $s_0 = s_f = v_f = 0$ and $v_0 = 1$, as in the numerical example studied in~\cite{BauBurKay2019}.  Then one has the case Theorem~\ref{control_crit}$(a)(i)$: Equation~\eqref{t_crit} reduces to
\[
-t_c^2 + \frac{1}{2} = 0\,,
\]
yielding $t_c = 1/\sqrt{2}$.  Then using \eqref{r2}, one gets $r = -1/(\sqrt{2} - 1) = -(1 + \sqrt{2})$, or $a_c = 1 + \sqrt{2}$.  Finally, the optimal control can simply be written from~\eqref{u_crit1} as
\begin{equation}  \label{u_crit_ex}
	u_c(t) = \left\{\begin{array}{rl}
	-(1 + \sqrt{2})\,, &\ \ \mbox{if\ \ } 0\le t < 1/\sqrt{2}\,, \\[1mm]
	1 + \sqrt{2}\,, &\ \ \mbox{if\ \ } 1/\sqrt{2}\le t \le 1\,.
	\end{array} \right.
\end{equation}
The numerical observations made in~\cite{BauBurKay2019}, re-iterated in~\eqref{observation1}--\eqref{observation2}, agree with the result in~\eqref{u_crit_ex}: $a_c \approx 2.4142$ and $t_c \approx 0.7071$, correct to four dp. \\
\end{remark}

\section{Numerical Experiments}
\label{sec:num_exp}

For computations numerically solving the three problems in Sections~\ref{ex:DI}--\ref{ex:mctool}, we employ the AMPL--Ipopt computational suite: AMPL is an optimization modelling language~\cite{AMPL} and Ipopt is an Interior Point Optimization software~\cite{WacBie2006} (version 3.12.13 is used here).  The suit is commonly utilized to solve discretized optimal control problems.  We discretize the optimal control problems~(Pf) and (Pcf) using the Euler scheme, with the number of  time discretization nodes (or time partition points) set in most of the cases as $2000$.  The number of these nodes is increased (as reported in situ) only when $a_c$ (in the case of critically infeasible solution) needs to be reported with a higher accuracy.  The Euler scheme is more suitable than higher-order Runge--Kutta discretization for these problems as the solutions exhibit bang--bang types of control making the state variable solutions only of $C^0$ class of functions.  Numerical chatter is evident when a higher-order discretization scheme, such as the trapezoidal rule, is used.  With $2000$ grid points, the resulting large-scale finite-dimensional problems have about $6000$ variables and $4000$ constraints for the double integrator and the damped oscillator problems, and $16000$ variables and $4000$ constraints for the machine tool manipulator problem.  We set the tolerance {\tt tol} for Ipopt to $10^{-8}$ in all problems.  We note that AMPL can also be paired with other optimization software, such as Knitro~\cite{Knitro}, SNOPT~\cite{GilMurSau2005} or TANGO~\cite{AndBirMarSch2008,BirMar2014}.

All three example problems in Sections~\ref{ex:DI}--\ref{ex:mctool} have a single control variable and the constraint on the control is given as
\[
-a \le u(t) \le a\,,\ \ \mbox{a.e.}\  t\in[0,1]\,,
\]
where $a$ is a positive constant.  The optimality condition~\eqref{uB} can then be written for this particular case as
\begin{equation}  \label{ex:uB}
	u_{\cal B}(t) = \left\{\begin{array}{rl}
	a\,, &\ \ \mbox{if\ \ } v(t) \ge 0\,, \\[1mm]
	-a\,, &\ \ \mbox{if\ \ } v(t) < 0\,,
	\end{array} \right.
\end{equation} 
for a.e.\ $t\in[0,1]$.  We will conveniently verify the optimality of the numerical results using~\eqref{ex:uB}.

\subsection{Double integrator}
\label{ex:DI}

From the double integrator problem~(PDI) in Section~\ref{sec:DI}, one simply has
\[
A = \left[\begin{array}{cc}
0 &\ 1 \\
0 &\ 0
\end{array}\right],\quad
b = \left[\begin{array}{c}
0 \\
1
\end{array}\right],
\]
using the notation in Problem~(Pf) in Section~\ref{sec:best_approx} and Problem~(Pcf) in Section~\ref{sec:crit_feas}.  As in Remark~\ref{control_crit_ex}, we take $s_0 = s_f = v_f = 0$ and $v_0 = 1$.  In other words, the boundary conditions in $\varphi(x(0),x(1)) = 0$ are expressed as $x(0) = (0, 1)$ and $x(1) = (0, 0)$.

First of all, we establish that the double integrator control system is controllable since $\rank Q_c = \rank [b\ |\ Ab] = 2 = n$.

\begin{figure}[t!]
\centering
\begin{subfigure}{.5\textwidth}
\centering
\includegraphics[width=\textwidth]{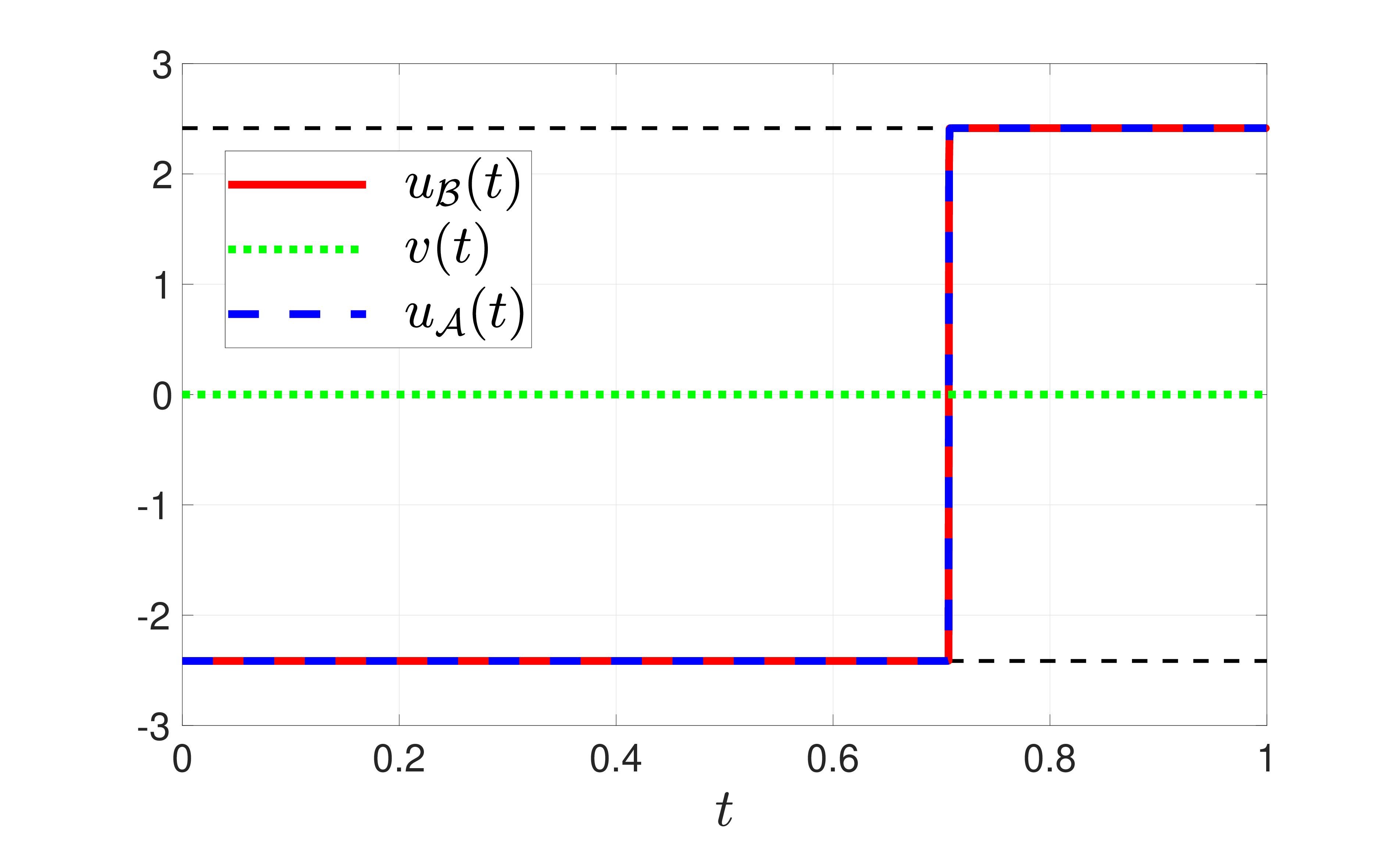}
\caption{$a = a_c \approx 2.414$}
\label{fig:a=2.414}
\end{subfigure}
\hfill
\begin{subfigure}{.49\textwidth}
\centering
\includegraphics[width=\textwidth]{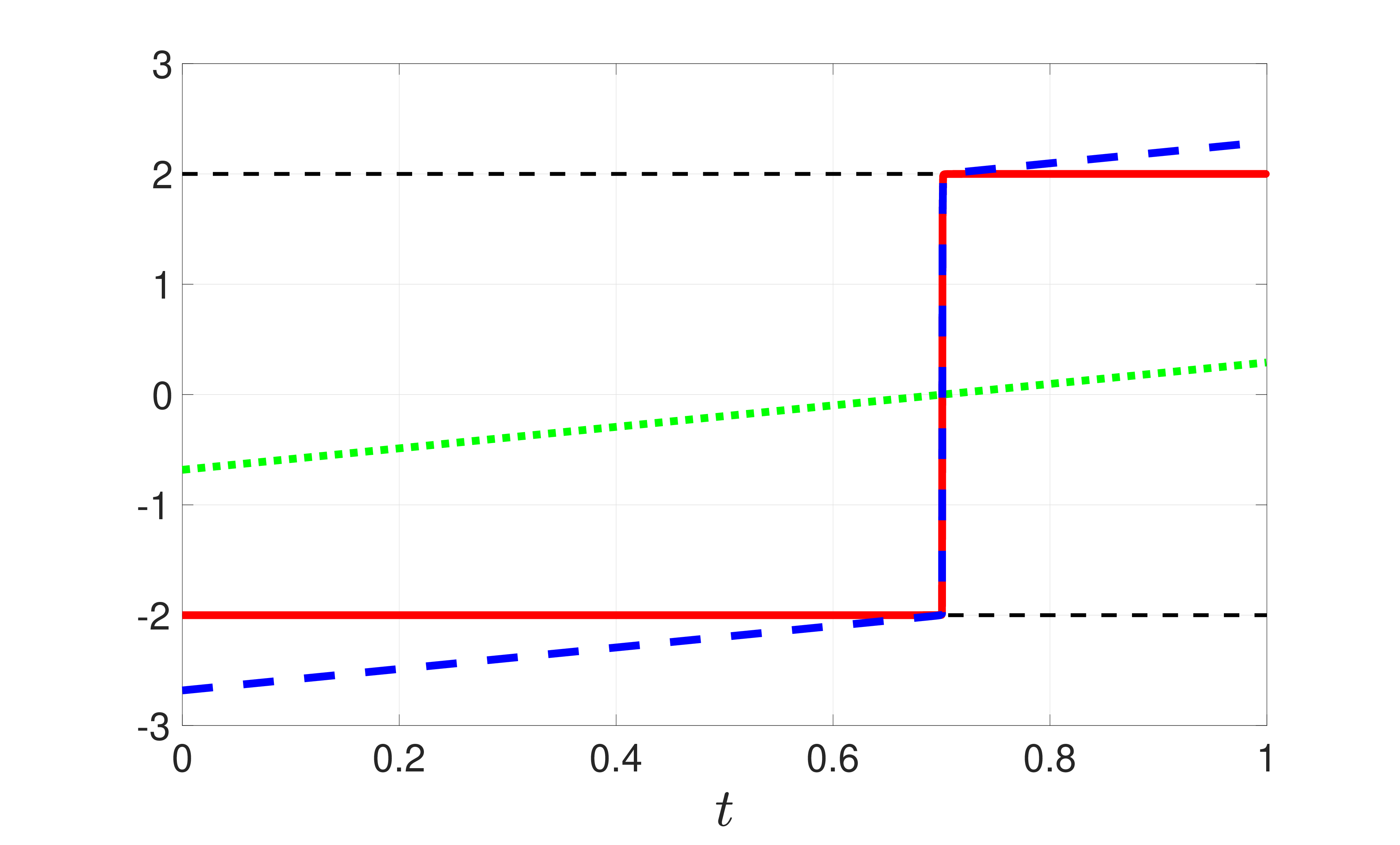}
\caption{$a = 2$}
\label{fig:a=2}
\end{subfigure}
\begin{subfigure}{.5\textwidth}
\centering
\includegraphics[width=\textwidth]{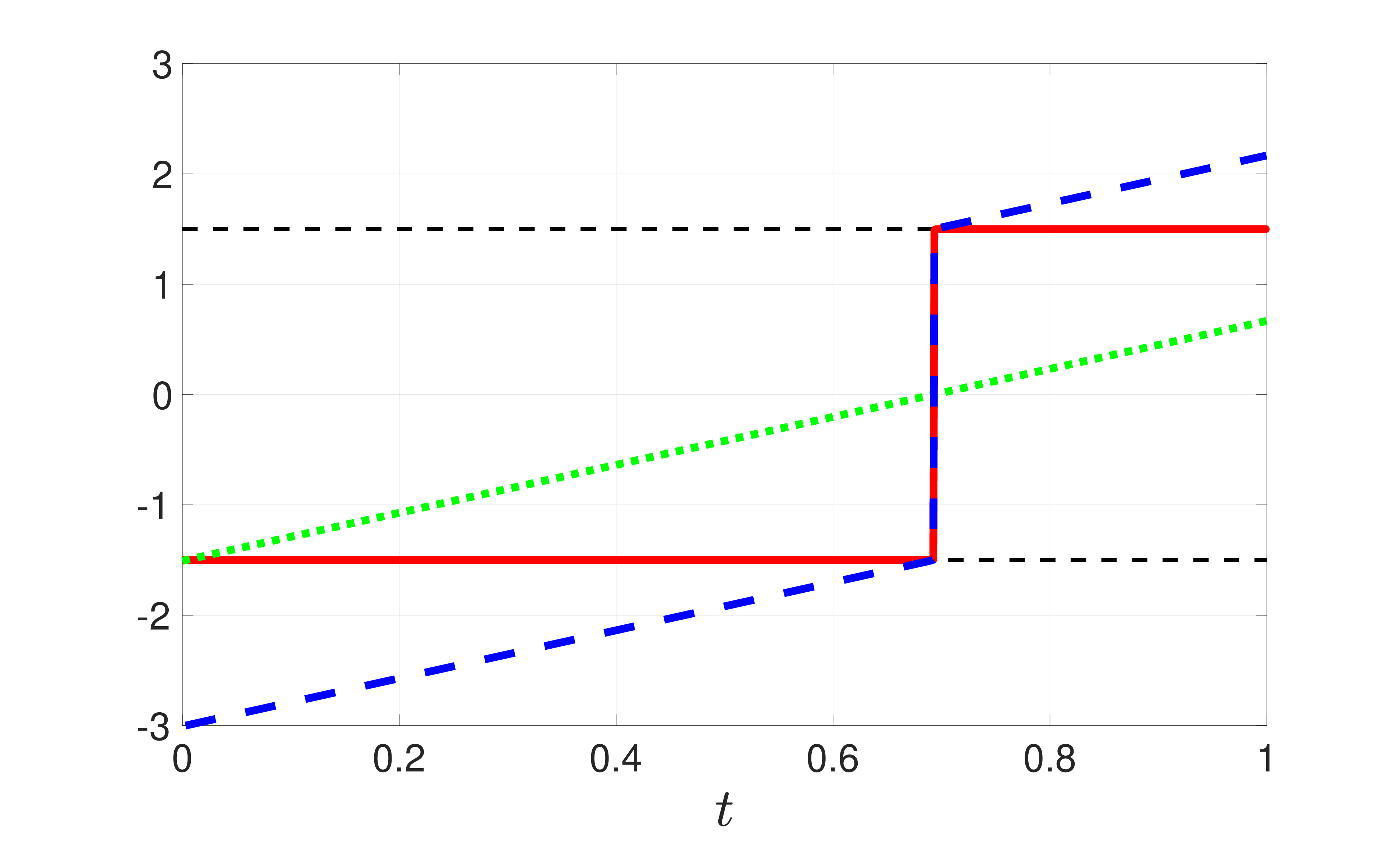}
\caption{$a = 1.5$}
\label{fig:a=1.5}
\end{subfigure}
\hfill
\begin{subfigure}{.49\textwidth}
\centering
\includegraphics[width=\textwidth]{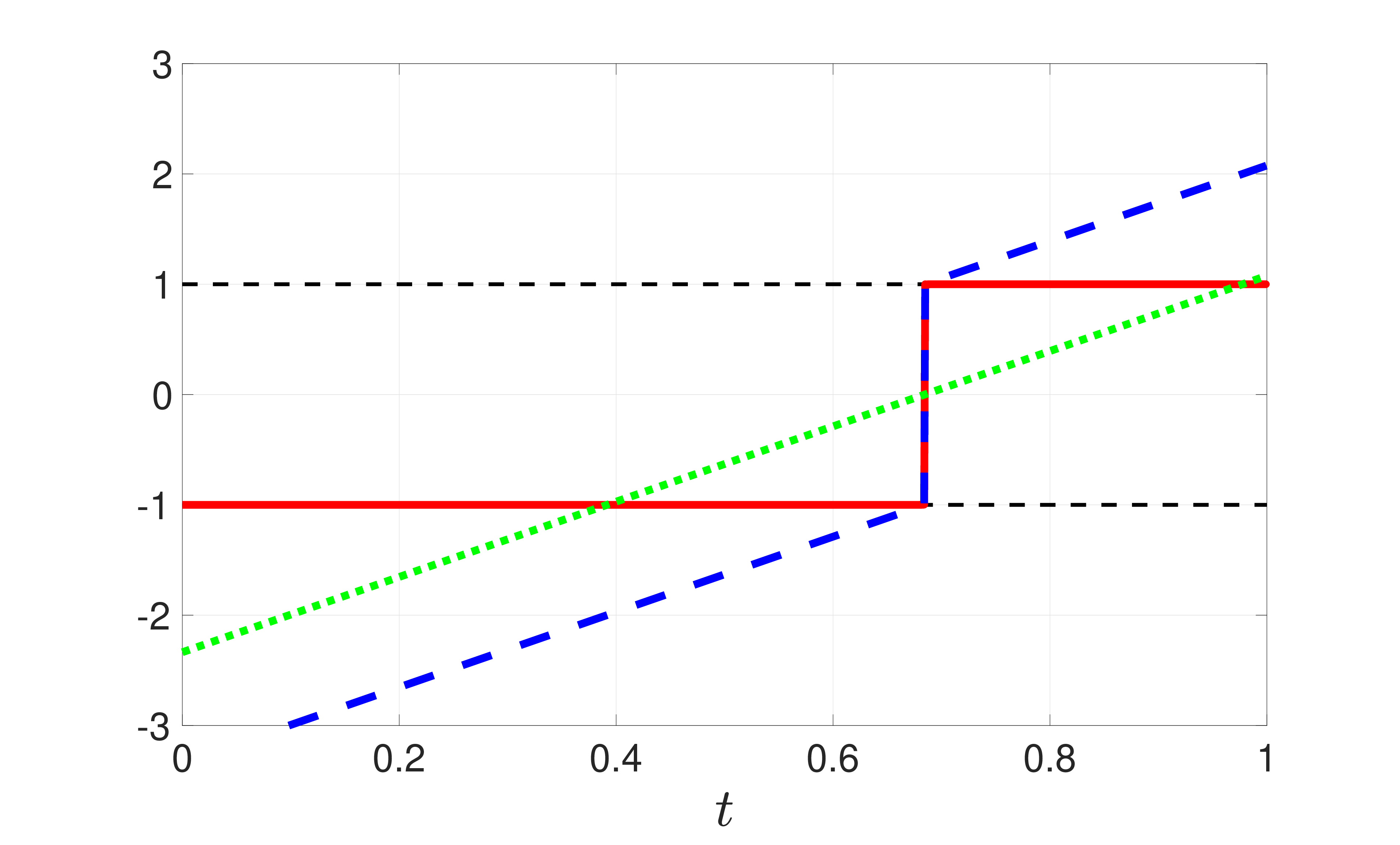}
\caption{$a = 1$}
\label{fig:a=1}
\end{subfigure}
\caption{\small\sf Double Integrator: (a) Critically feasible solution and (b)--(d) Best approximation solutions (infeasible case).}
\label{fig:DI}
\end{figure}

We have solved Problem~(Pcf) to find the critically feasible solution depicted in Figure~\ref{fig:DI}, where the solution curves for $u_{\cal A}$, $u_{\cal B}$ and $v$ are graphed.  With $10000$ time partition points, we obtained $a_c \approx 2.414$ ($2000$ time partition points only yields $a_c \approx 2.4$), which reconfirms the analytical solution $a_c = 1 + \sqrt{2} \approx 2.4142$ that was reported in Remark~\ref{control_crit_ex}.  We also observe (after zooming into the plot) that $t_c = 0.707$  which agrees with $t_c = 1/\sqrt{2} \approx 0.7071$ in Remark~\ref{control_crit_ex} up to three~dp.  The control $u_{\cal A}$ overlaps $u_{\cal B}$ since, in the critically feasible case, ${\cal A}\cap{\cal B} \neq\emptyset$ and so the gap function $v$ is the zero function.  The graph of $u = u_{\cal A} = u_{\cal B}$ in Figure~\ref{fig:a=2.414} in turn verifies the analytical expression in~\eqref{u_crit_ex}.  

For the infeasible case, i.e., when $a < a_c$, it is no longer possible to get a solution analytically, even for the relatively simple-looking double integrator problem.  In Figures~\ref{fig:a=2}--\ref{fig:a=1}, the solution plots for $a=2, 1.5$ and $1$ are shown, respectively.  The solution for $u_{\cal B}$ is of bang--bang type with one switching, verifying Corollary~\ref{cor: infeas_DI}.  The role of $v$ as a switching function is clear from these plots.  We recall the fact that $v = -\lambda_2$ by Theorem~\ref{thm:gap&bbang}, and point that $v$ appears linearly in the plots since $\lambda(t)$ is linear in $t$. The switching time for each case shown in Figures~\ref{fig:a=2}--\ref{fig:a=1} is found graphically as: (b) $t_s \approx 0.701$, (c) $t_s \approx 0.693$ and (d) $t_s \approx 0.685$.  Further numerical experiments with even smaller $a$ suggest that as $a\to 0$, $t_s\to 2/3$.

\subsection{Damped oscillator}
\label{ex:DO}

While the ODE underlying the double integrator problem is $\ddot{z}(t) = u(t)$, the ODE underlying the damped oscillator problem is $\ddot{z}(t) + 2\,\zeta\,\omega_n\,\dot{z}(t) + \omega_n^2(t)\,z(t) = u(t)$, with the damping and stiffness terms added, where the parameter $\omega_n>0$ is the natural frequency and the parameter $\zeta\ge 0$ is the damping ratio of the system.  When $\zeta = 0$ the system is referred to as the (simple) harmonic oscillator.  Defining the state variables $x_1 := z$ and $x_2 := \dot{z}$ (as in the case of the double integrator), one gets, for the case of the damped oscillator,
\[
A = \left[\begin{array}{cc}
0 &\ 1 \\
-\omega_n^2 &\ -2\,\zeta\,\omega_n 
\end{array}\right],\quad
b = \left[\begin{array}{c}
0 \\
1
\end{array}\right],
\]
again using the notation in Problems~(Pf) and~(Pcf). As the time interval of the problem we take $[0,1]$, and set the boundary conditions to be the same as those of the double integrator problem: $x(0) = (0, 1)$ and $x(1) = (0, 0)$.  We set the values of the parameters as $\omega_n = 20$ and $\zeta = 0.1$.

First, we can assert that the damped oscillator control system is controllable since $\rank Q_c = \rank [b\ |\ Ab] = 2 = n$.

\begin{figure}[t!]
\centering
\begin{subfigure}{.5\textwidth}
\centering
\includegraphics[width=\textwidth]{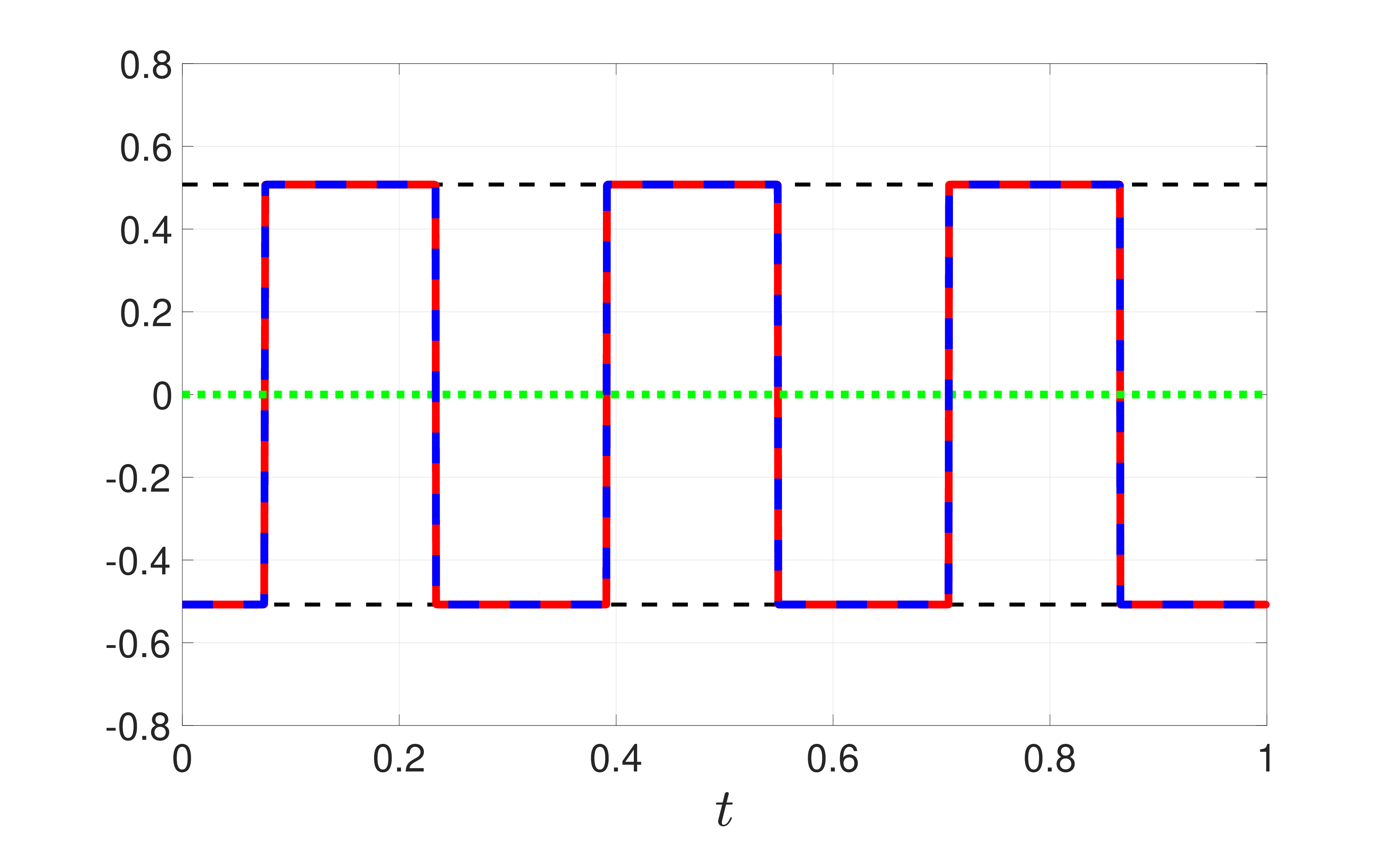}
\caption{$a = a_c \approx 0.475$}
\label{fig:a=0.475}
\end{subfigure}
\hfill
\begin{subfigure}{.49\textwidth}
\centering
\includegraphics[width=\textwidth]{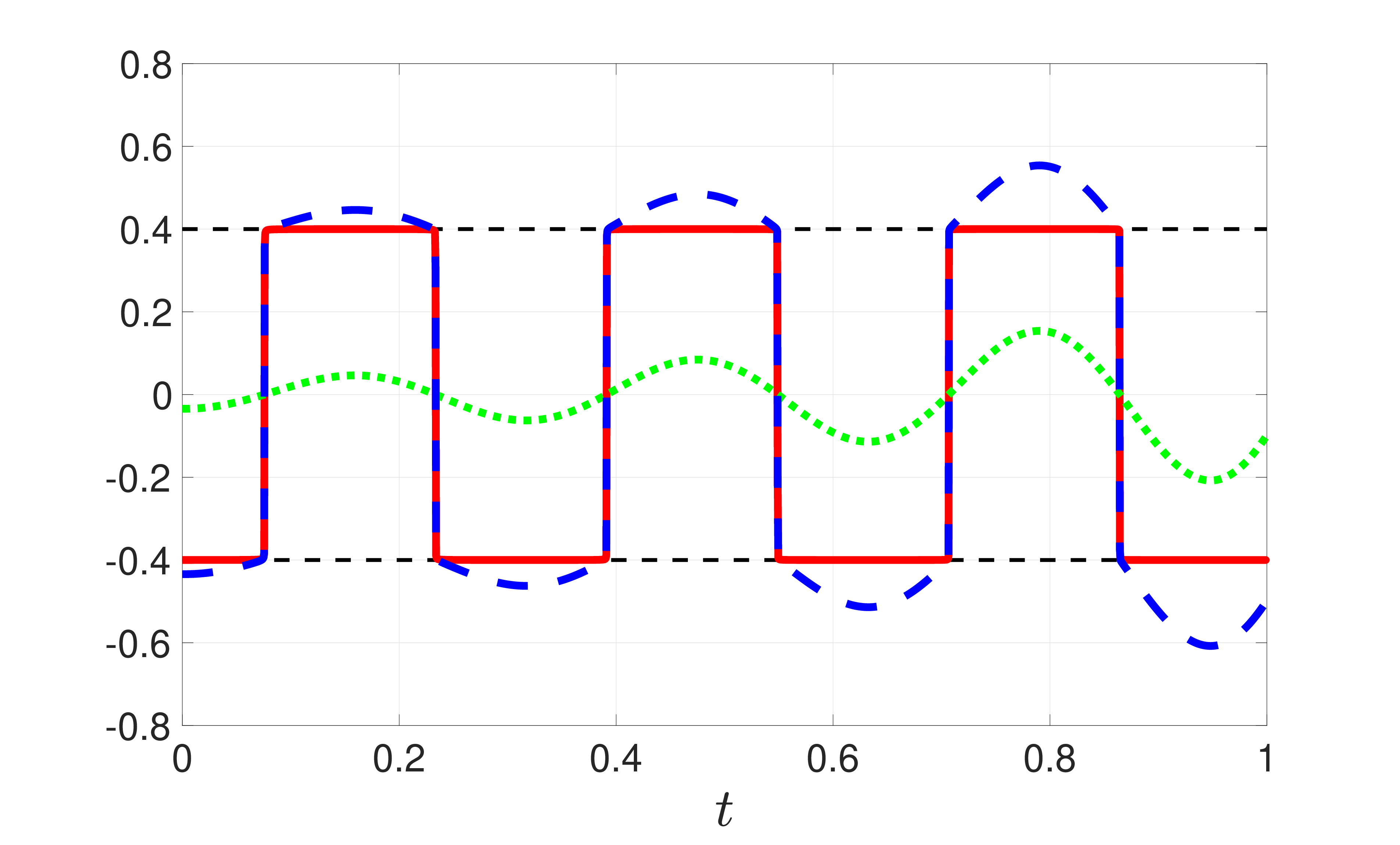}
\caption{$a = 0.4$}
\label{fig:a=0.4}
\end{subfigure}
\begin{subfigure}{.5\textwidth}
\centering
\includegraphics[width=\textwidth]{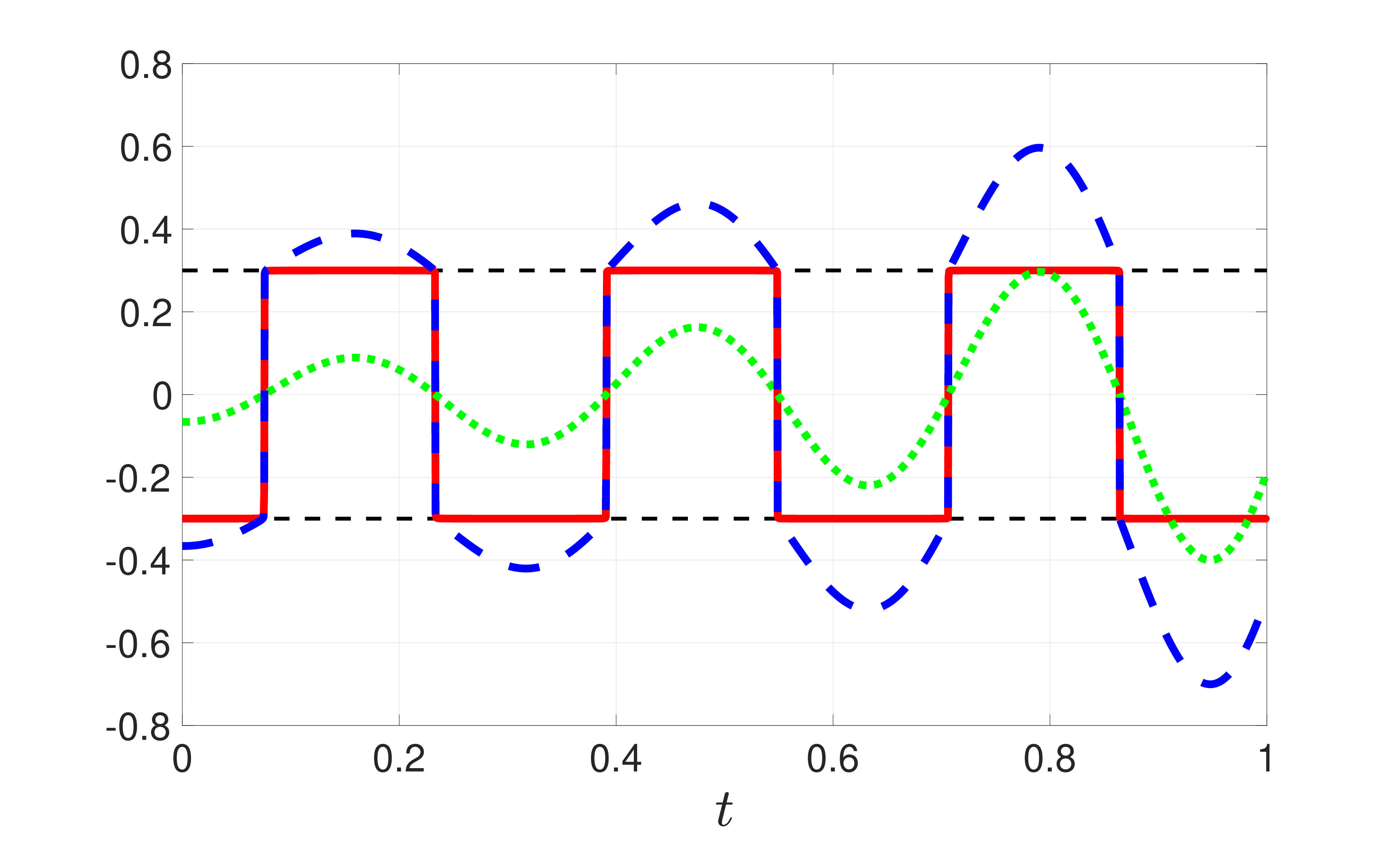}
\caption{$a = 0.3$}
\label{fig:a=0.3}
\end{subfigure}
\hfill
\begin{subfigure}{.49\textwidth}
\centering
\includegraphics[width=\textwidth]{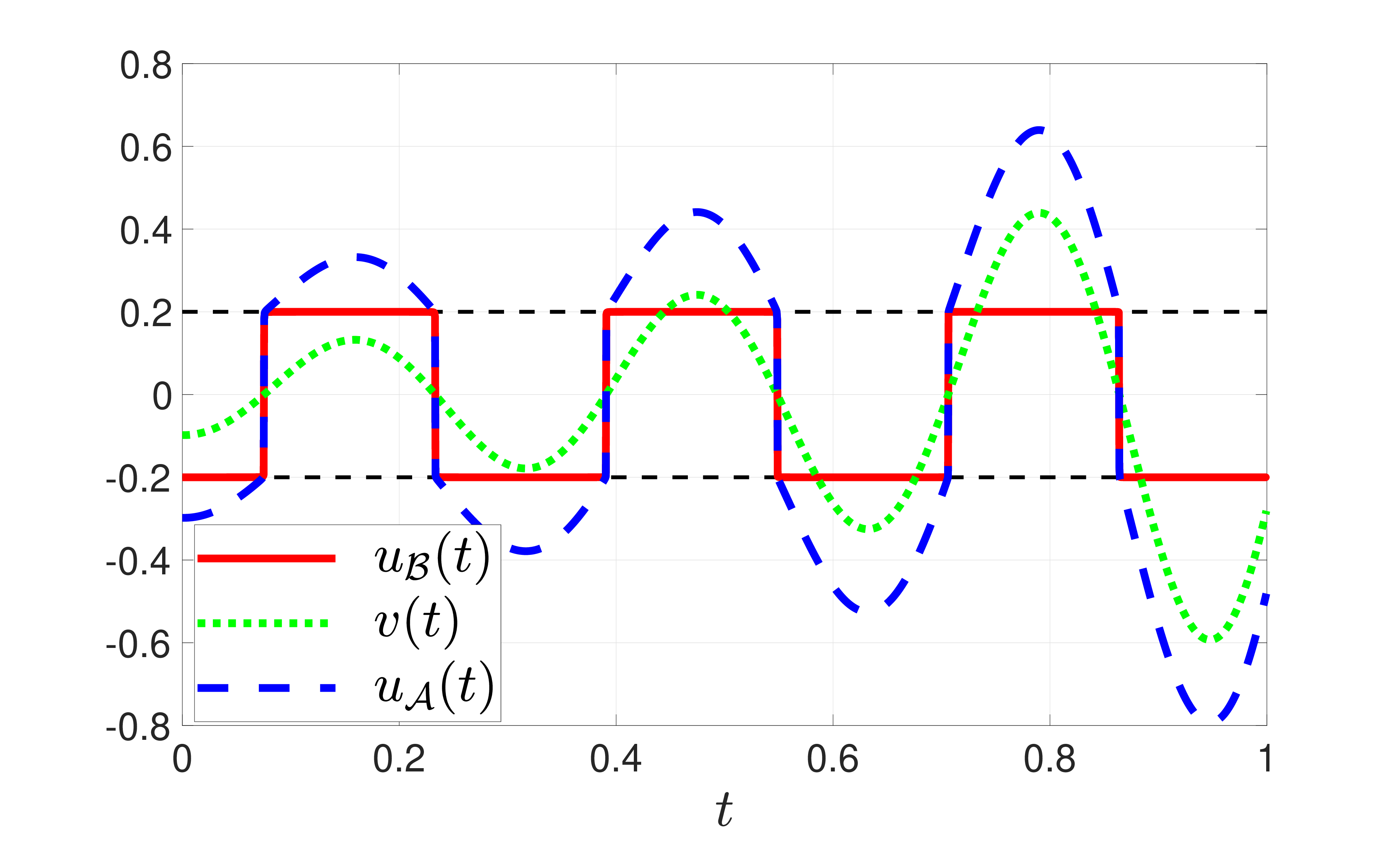}
\caption{$a = 0.2$}
\label{fig:a=0.2}
\end{subfigure}
\caption{\small\sf Damped oscillator: (a) Critically feasible solution and (b)--(d) Best approximation solutions
(infeasible case).}
\label{fig:DO}
\end{figure}

Numerical solutions to Problems~(Pf) and~(Pcf) are depicted in Figure~\ref{fig:DO}:  The critically feasible solution to (Pcf) appears in Figure~\ref{fig:a=0.475} and the infeasible solutions to (Pf) appear in Figures~\ref{fig:a=0.4}--\ref{fig:a=0.2}.
With $2\times 10^5$ time partition points, we have obtained $a_c \approx 0.475$, correct to three dp.  No analytical solution is available.  As expected from Theorem~\ref{theo:crit}, the control $u_{\cal B}$ is of bang--bang type, and it overlaps with $u_{\cal A}$.  The control $u_{\cal B}$ appears to be periodic with six switchings.  Further experiments with various other boundary conditions not only result in different $a_c$ but also in different number of switchings; but the control $u_{\cal B}$ still appears to be periodic.

In Figures~\ref{fig:a=0.4}--\ref{fig:a=0.2}, we provide the respective solution plots for $a=0.4, 0.3$ and $0.2$.  The solution for $u_{\cal B}$ is of bang--bang type as asserted by Theorem~\ref{thm:gap&bbang}.  It is observed that not only the control $u_{\cal B}$ appears to be periodic but also the switching times seem to remain the same as those in the critically feasible solution in Figure~\ref{fig:a=0.475}.  The role of the gap function $v$ as a switching function is clear from these plots, verifying \eqref{ex:uB}.

\subsection{Machine tool manipulator}
\label{ex:mctool}

A linear ODE model and an associated optimal control problem for a machine tool manipulator is described in~\cite{ChrMauZir2010}.  Using the notation in Problems~(Pf) and~(Pcf), one has that
\begin{align*}
A &= \begin{bmatrix}
0 & 0 & 0 & 1 & 0 & 0 & 0 \\
0 & 0 & 0 & 0 & 1 & 0 & 0 \\
0 & 0 & 0 & 0 & 0 & 1 & 0 \\
-4.441\times10^7/450 & 0 & 0 & -8500/450 & 0 & 0 & -1/450 \\
0 & 0 & 0 & 0 & 0 & 0 & 1/750 \\
0 & 0 & -8.2\times10^6/40 & 0 & 0 & -1800/40 & 0.25/40 \\
0 & 0 & 0 & 0 & 0 & 0 & -1/0.0025
\end{bmatrix}, \\ b &= \begin{bmatrix}
0 & 0 & 0 & 0 & 0 & 0 & 1/0.0025
\end{bmatrix}^T.
\end{align*}
Clearly, the control system has seven state variables and one control variable.  In~\cite{ChrMauZir2010}, the time interval for the dynamics is chosen to be $[0,0.0522]$, and the boundary conditions are imposed as $x(0)=(0,0,0,0,0,0,0)$, $x(0.0522)=(0,0.0027,0,0,0.1,0,0)$.  Moreover, the control variable is constrained as $-2000 \le u(t) \le 2000$, under which the problem is feasible.  A minimum-energy control model for this machine tool manipulator has also subsequently been studied in~\cite{BurKayMaj2014, BurCalKayMou2023, BurCalKay2024}.

It can easily be verified that the machine tool manipulator control system is controllable as $\rank Q_c = \rank [b\ |\ Ab\ |\ A^2b\ |\ \cdots |\ A^6b] = 7 = n$.

\begin{figure}[t!]
\centering
\begin{subfigure}{.5\textwidth}
\centering
\includegraphics[width=\textwidth]{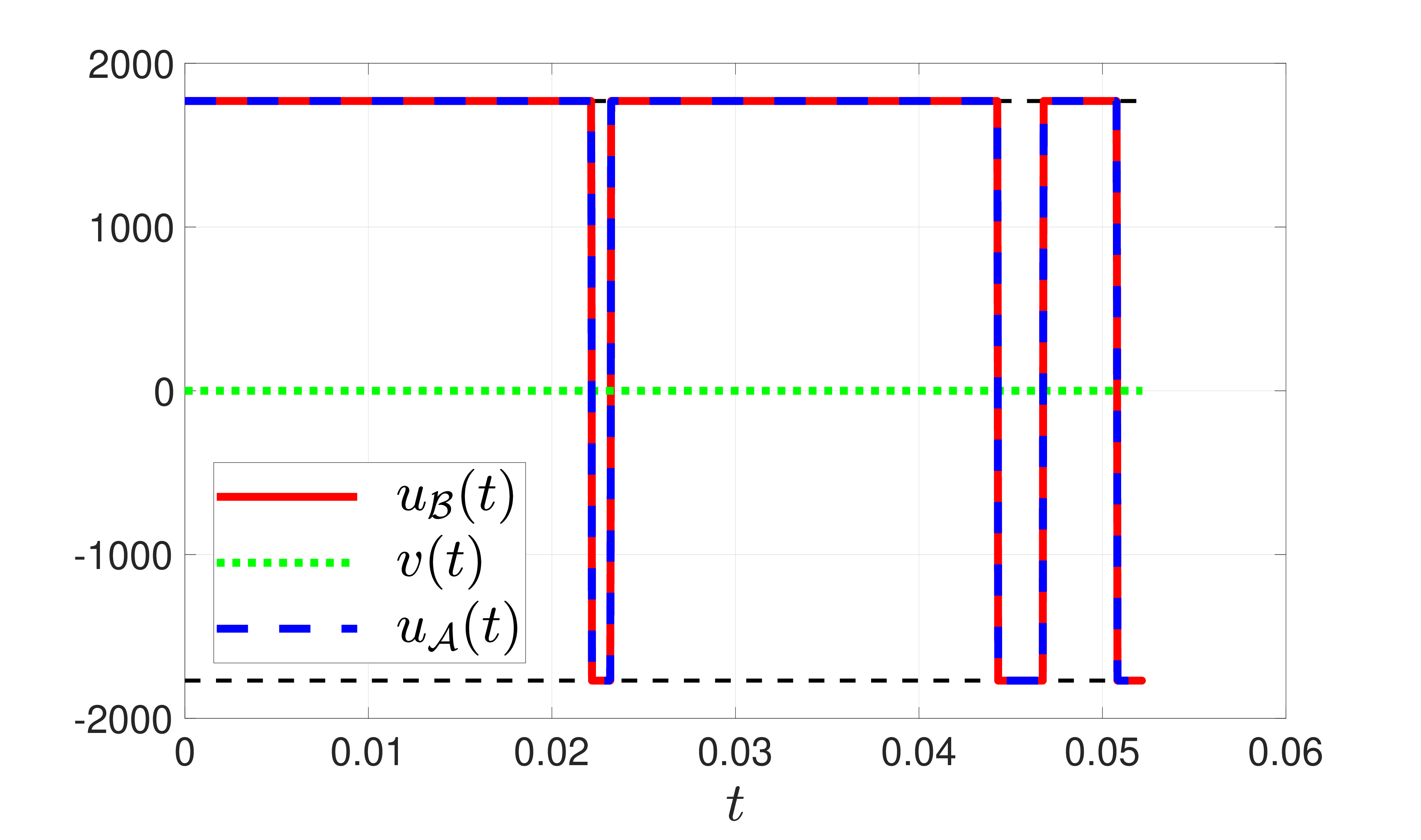}
\caption{$a = a_c \approx 1769.46$}
\label{fig:a=1769.46}
\end{subfigure}
\hfill
\begin{subfigure}{.49\textwidth}
\centering
\includegraphics[width=\textwidth]{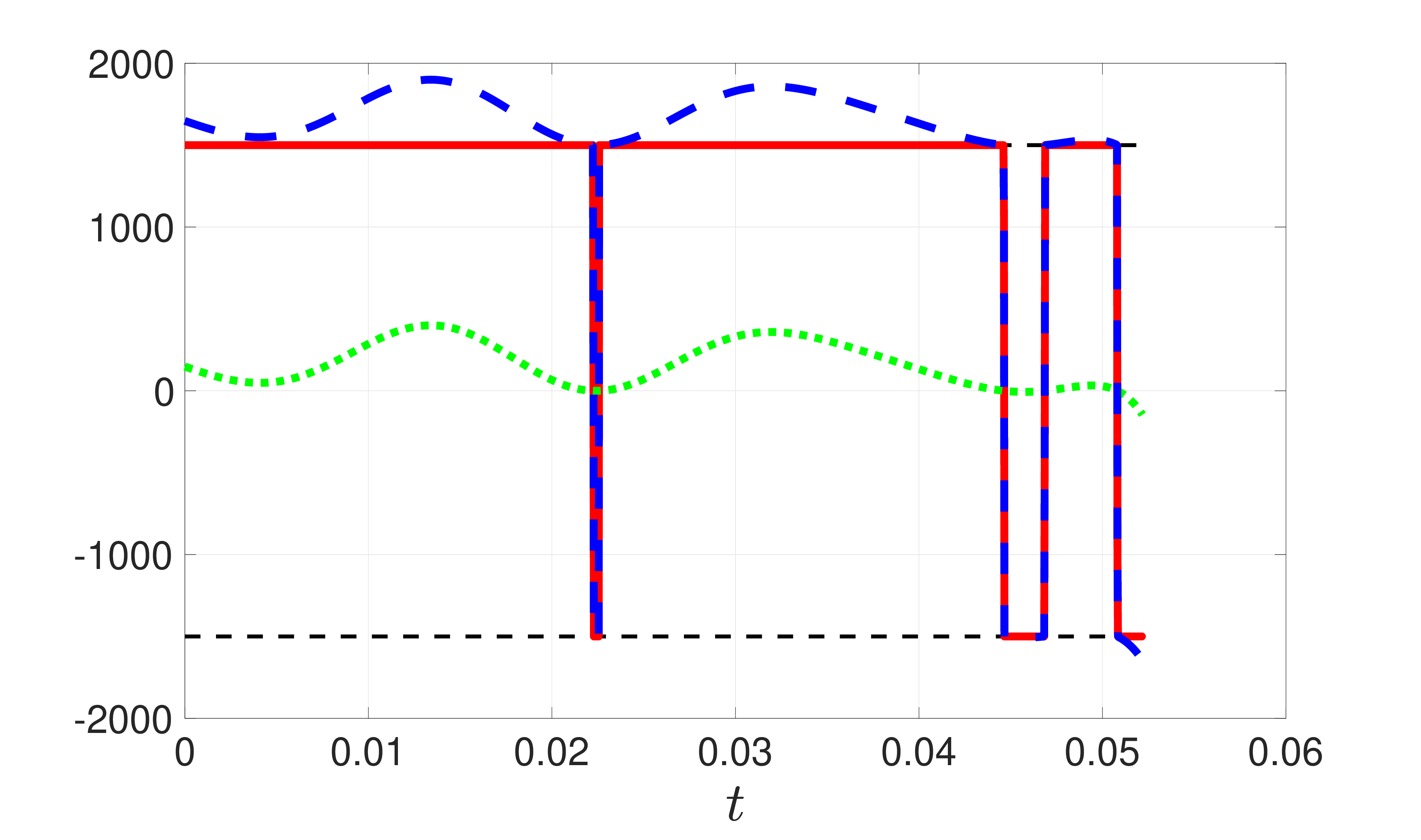}
\caption{$a = 1500$}
\label{fig:a=1500}
\end{subfigure}
\begin{subfigure}{.5\textwidth}
\centering
\includegraphics[width=\textwidth]{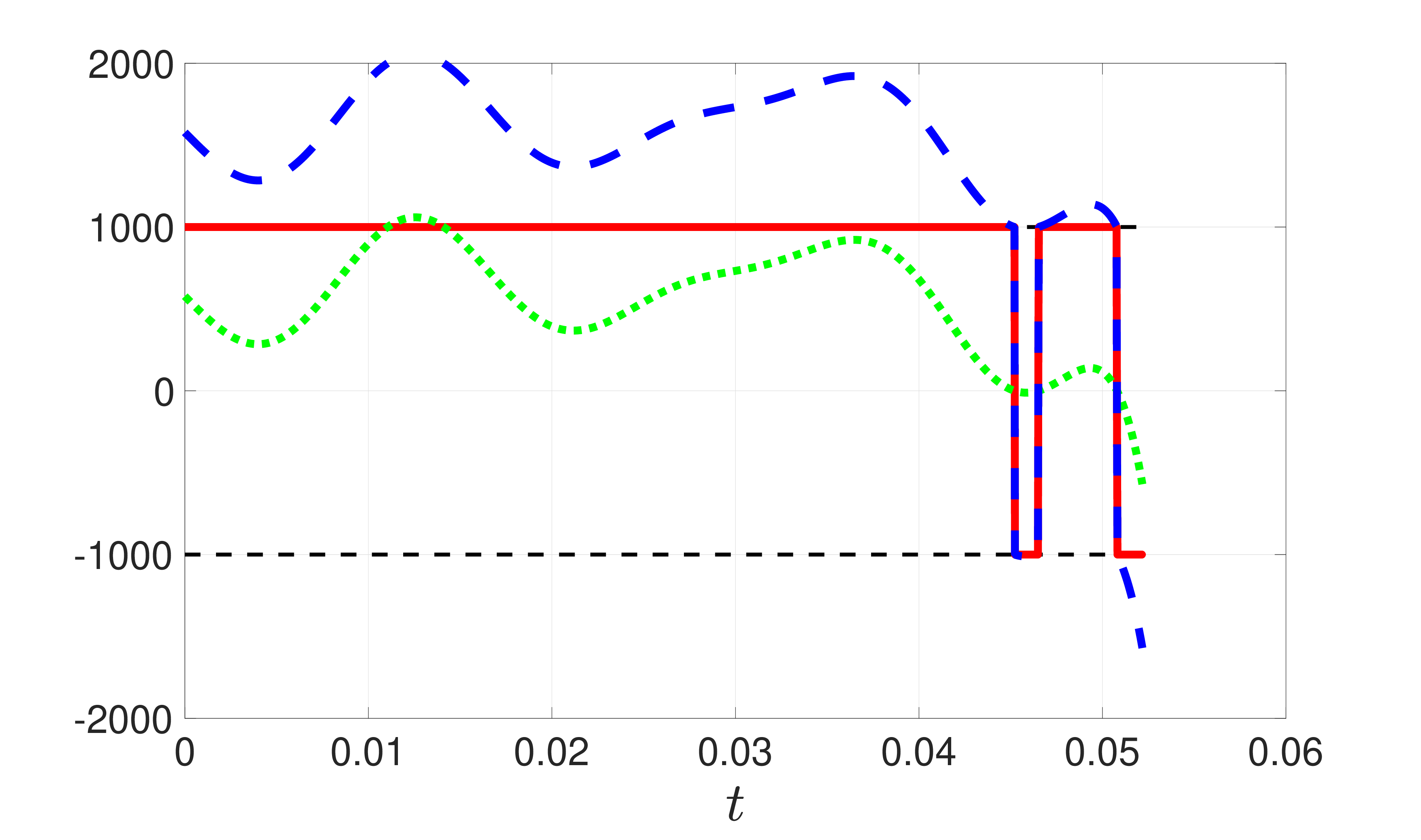}
\caption{$a = 1000$}
\label{fig:a=1000}
\end{subfigure}
\hfill
\begin{subfigure}{.49\textwidth}
\centering
\includegraphics[width=\textwidth]{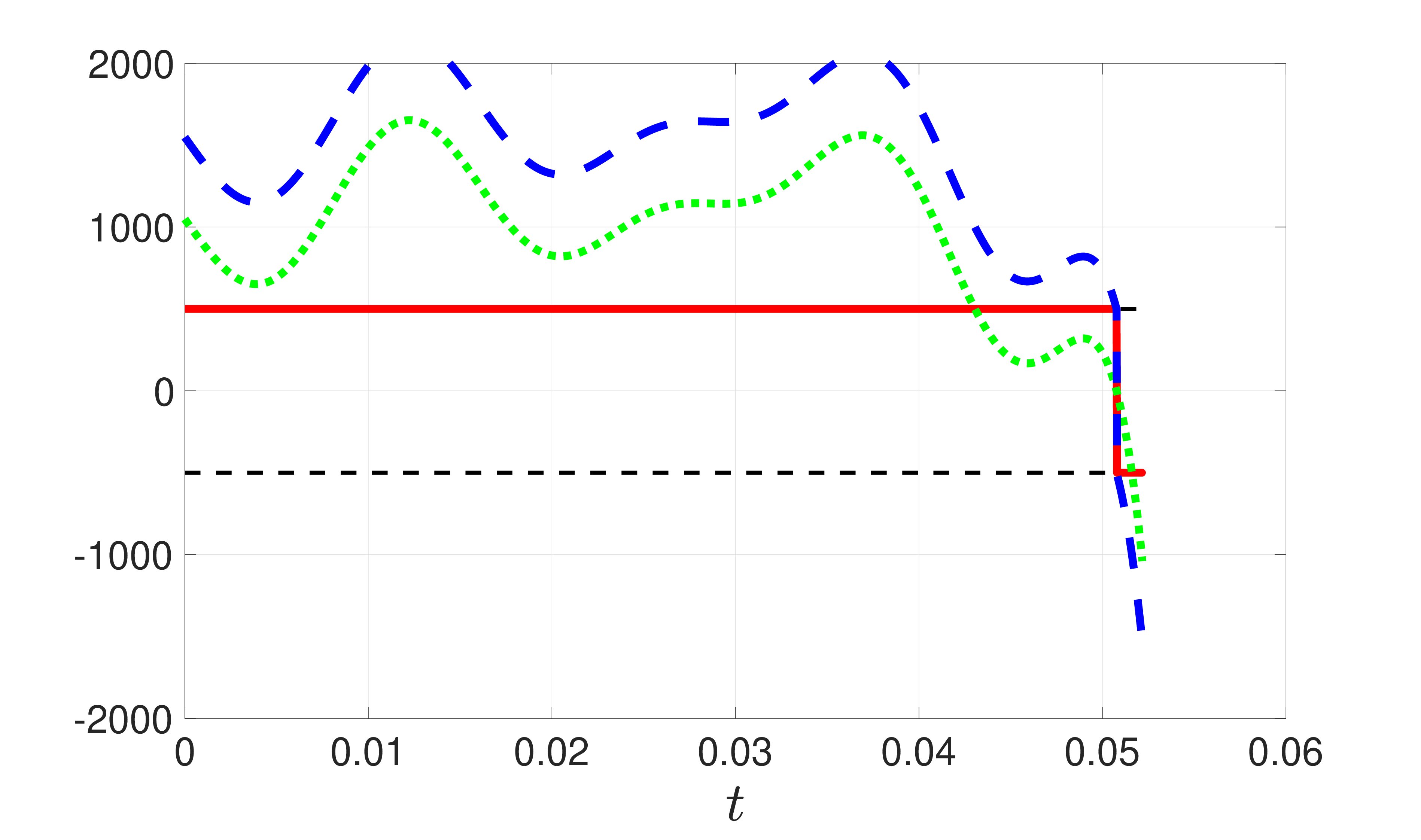}
\caption{$a = 500$}
\label{fig:a=500}
\end{subfigure}
\caption{\small\sf Machine tool manipulator: (a) Critically feasible solution and (b)--(d) Best approximation solutions (infeasible case).}
\label{fig:mctool}
\end{figure}

Numerical solutions to Problems~(Pf) and~(Pcf) are depicted in Figure~\ref{fig:mctool}:  The critically feasible solution to (Pcf) is depicted in Figure~\ref{fig:a=1769.46} and the infeasible solutions to (Pf) appear in Figures~\ref{fig:a=1500}--\ref{fig:a=500}.  With $10000$ time partition points, and implementing SNOPT (instead of Ipopt) with AMPL, we obtained $a_c \approx 1769.46$---Just on this occasion Ipopt was not successful in getting a solution.  As asserted by Theorem~\ref{theo:crit}, the control $u_{\cal B}$ is of bang--bang type, and it overlaps with $u_{\cal A}$.  The control $u_{\cal B}$ appears to have five switchings.

In Figures~\ref{fig:a=1500}--\ref{fig:a=500}, we provide the solutions for $a=1500, 1000$ and $500$.  The solution for $u_{\cal B}$ is of bang--bang type as asserted by Theorem~\ref{thm:gap&bbang}.  We observe that the number of switchings decreases with decreasing $a$: With $a = 500$, and by other experiments with $a < 500$, numerical solutions suggest that there is only one switching.  The role of the gap function $v$ as a switching function is clear from these plots for this example as well, verifying \eqref{ex:uB}.

\section{Conclusion}
\label{sec:conc}
We have studied a class of infeasible and critically feasible optimal control problems and proved that the best approximation control in the box constraint set is of bang--bang type for each problem.  We presented a full analytical solution for the critically feasible double integrator problem.  We numerically illustrated these results on three increasingly difficult example problems.  For numerical computations, we discretized the example problems and solved large-scale optimization problems using popular optimization software.

The numerical scheme described in this paper can further be improved:  Since the solution structure is known to be of bang--bang type, one can solve problems discretized over a coarse time grid first, and then, once there is a rough idea about the number of switchings and the places of the switchings, a switching time parameterization technique (see \cite{KayNoa2013, MauBueKimKay2005, OsmMau2012}) can  be implemented to find the switching times accurately.

The paper~\cite{BauMou2023} motivated us in looking at infeasible optimal control problems and study the properties of the gap (function) vector.  Reference~\cite{BauMou2023} also studies in a theoretical setting an application of the Douglas--Rachford algorithm to infinite-dimensional infeasible optimization problems in Hilbert space.  A next step would be to employ the Douglas--Rachford algorithm to solve the infeasible optimal control problems we are looking at in the present paper.  It would also be interesting to employ and test the Peaceman--Rachford algorithm~\cite[Section 26.4 and Proposition~28.8]{BC2017}, which is another projection type method, for the class of problems we have studied.

It would be interesting to extend the applications in this paper to the case of infeasible and critically infeasible nonconvex optimal control problems, including those with state constraints, and carry out numerical experiments, although no theory is available yet for such more general classes of problems in infinite dimensions.

\end{document}